\documentclass[]{amsart}  

\usepackage[T1]{fontenc}
\usepackage{Nyquist_sty}
\usepackage{xcolor}
\usepackage{subcaption}

\usepackage{mathtools}
\usepackage{bbm}
\usepackage{booktabs}
\usepackage{lmodern}
\usepackage[disable]{todonotes}
\usepackage[unicode]{hyperref}
\usepackage[normalem]{ulem}

\PassOptionsToPackage{unicode}{hyperref}

\usepackage{tikz}
\usetikzlibrary{tikzmark, shapes.geometric, shapes.arrows, fit, bending}

\numberwithin{equation}{section}

\newcommand{\close}[1]{\overline{#1}}
\newcommand{\interior}[1]{#1^\mathrm{o}}
\newcommand\restr[2]{{%
  \left.\kern-\nulldelimiterspace %
  #1 %
  \vphantom{\big|} %
  \right|_{#2} %
  }}
\newcommand{\spt}{\operatorname{spt}}

\makeatletter
\newcommand*\bigcdot{\mathpalette\bigcdot@{.5}}
\newcommand*\bigcdot@[2]{\mathbin{\vcenter{\hbox{\scalebox{#2}{$\m@th#1\bullet$}}}}}
\makeatother
\renewcommand{\cdot}{\bigcdot}

 \newtheorem{taggedassumptionx}{Assumption}
\newenvironment{taggedassumption}[1]
 {\renewcommand\thetaggedassumptionx{#1}\taggedassumptionx}
 {\endtaggedassumptionx}

\begin{document}   

\title[Large deviations of the dynamic Schr\"odinger problem]{A weak convergence approach to the large deviations of the dynamic Schr\"odinger problem} 

\author[V.~Nilsson]{Viktor Nilsson}
\author[P.~Nyquist]{Pierre Nyquist}
\address[V.~Nilsson, P.\ Nyquist]{KTH Royal Institute of Technology}
\address[P.\ Nyquist]{Chalmers University of Technology and University of Gothenburg}
\email{{vikn@kth.se}, {pnyquist@chalmers.se}}
\thanks{}



\begin{abstract} 
In this paper, we consider the large deviations for dynamical Schr\"odinger problems, using the variational approach developed by Dupuis, Ellis, Budhiraja, and others.
Recent results on scaled families of Schr\"odinger problems, in particular by Bernton, Ghosal, and Nutz \cite{bernton2022entropic}, and the authors \cite{nilsson2025large}, have established large deviation principles for the static problem. 
For the dynamic problem, only the case with a scaled Brownian motion reference process has been explored by Kato in \cite{Kato24}. %

Here, we derive large deviations results using the variational approach, with the aim of going beyond the Brownian reference dynamics considered in \cite{Kato24}. %
Specifically, we develop a uniform Laplace principle for bridge processes conditioned on their endpoints. 
When combined with existing results for the static problem, this leads to a large deviation principle for the corresponding (dynamic) Schr\"odinger bridge. 
In addition to the specific results of the paper, our work puts such large deviation questions into the weak convergence framework, and we conjecture that the results can be extended to cover also more involved types of reference dynamics. %
Specifically, we provide an outlook on applying the result to reflected Schr\"odinger bridges.

\end{abstract}   

\maketitle

\section{Introduction}
\label{sec:intro}
In recent years, optimal transport (OT) has gained increasing importance, both within mathematics and in applications, primarily in the data science and machine learning contexts; see for example \cite{peyre2019, nutz2021introduction} and the references therein. 
In applications, (almost) solving OT problems has become tractable in more and more situations due to algorithmic advances, such as Sinkhorn's algorithm for discrete cases, and diffusion Schr\"odinger bridges in the continuous case. These methods rely on an entropic regularization of the OT problem, resulting in an \emph{entropic optimal transport} (EOT) problem. From a different perspective, these problems may also be seen as static \emph{Schr\"odinger bridges} (SB).

In order to understand of how close EOT/SB plans are to the corresponding (i.e., with no regularization) OT plan, we must understand the impact of the regularization parameter and how the plans converge as this parameter decreases to zero. One way of obtaining such an understanding is to establish rates of convergence of the EOT/SB plans to the OT plan(s).
This has recently been done by applying the framework of \emph{large deviations}, first in \cite{bernton2022entropic} for the EOT problem with a fixed cost function, and then extended to a class of SB problems in \cite{nilsson2025large}. In the latter, \emph{large deviation principles} (LDP) are developed for \emph{static} SBs, and the corresponding EOT plans, under certain conditions on the reference dynamics, or equivalently the associated cost functions (see \cite{nilsson2025large} for the details). In \cite{Kato24}, the large deviation results of \cite{bernton2022entropic} are extended to \emph{dynamical} SBs for the special case of Brownian reference dynamics; recently, Kato also extended his own results to cover Langevin dynamics when the associated potential is bounded and has bounded derivatives (see Proposition 3.3 in \cite{Kato24}) by applying an argument similar to the more general approach (for static problems) of \cite{nilsson2025large}. 

The main result of this paper can be viewed as an extension of the results in \cite{Kato24} to more general reference dynamics, as well as recasting the underlying question of large deviatons for Schr\"odinger problems in the framework of the \emph{weak convergence approach} of Dupuis, Ellis, Budhiraja and others \cite{Dupuis97, BD19}.

The remainder of the paper is organized as follows. In Section \ref{sec:main} we outline the main results, specifically the uniform Laplace principle for bridge processes (Theorem \ref{thm:uniform-laplace-principle}) and the LDP for dynamic Schr\"odinger bridges (Theorem \ref{thm:LDP_dSB}). In Section \ref{sec:prelim} we provide some preliminaries on Schr\"odinger bridge problems and large deviations. Section \ref{sec:development} is aimed at deriving the uniform Laplace principle for the type of reference SDE dynamics considered here. In Section \ref{sec:fullLDP} we derive the LDP for the relevant dynamic Schr\"odinger bridges, proving Theorem \ref{thm:LDP_dSB}. The paper concludes with somee brief examples in Section \ref{sec:case}.

\subsection{Main result}
\label{sec:main}

Before going into the details, we here outline the main results of the paper.
In the dynamic SB Problem, one starts with a \emph{reference measure} $R \in \calP(\calC_1)$, the space of probability measures on $\calC_1 \coloneq C([0,1]: \bR^d)$, describing how a typical particle travels; we will focus on reference dynamics $R$ associated with a class of SDEs (see below). %
The \emph{dynamic Schr\"{o}dinger bridge} with respect to $R$, $\mu$, and $\nu$ is defined as 
\begin{equation}\label{eq:dynamic-SB-def}
    \hat{\pi} \coloneq \argmin_{\pi \in \Pi^{\calC_1}(\mu, \nu)} \calH(\pi \mid\mid R),
\end{equation}
where $\calH(\cdot \mid\mid \cdot)$ is the relative entropy, or Kullback-Liebler (KL) divergence, and we define the set of path space couplings $\Pi^{\calC_1}(\mu, \nu)$ by
\begin{equation*}
    \Pi^{\calC_1} (\mu, \nu) \coloneq \left\{ \pi \in \calP (\calC _1): \pi_0 = \mu, \ \pi_1 = \nu \right\}.
\end{equation*}
By the strict convexity property of $\calH (\cdot \mid\mid R)$, $\hat \pi$ being well-defined follows from assuming the existence of any coupling $\pi \in \Pi^{\calC_1} (\mu, \nu)$ such that $\calH (\pi \mid\mid R) < \infty$. We tacitly assume that all considered SBs satisfy this condition, although there are useful generalizations like cyclical invariance, see, e.g., \cite{bernton2022entropic, nilsson2025large}.

Throughout we focus on establishing \textit{full} LDPs for dynamic SBs, which requires that this is also true for the underlying static problem. That is, the sequence of solutions $\{\pi_{\eta, 01}\} _\eta$ in the static problem satisfies a full LDP (see Section \ref{sec:prelim} for the notation). The simplest way to obtain this is to assume that the supports of initial and final time distributions $\mu$ and $\nu$ are compact \cite{bernton2022entropic, Kato24, nilsson2025large}.
\begin{assumption}
    The two marginals $\mu$, $\nu$ in the dynamic Schr\"odinger problem \eqref{eq:dynamic-SB-def} have compact supports.
\end{assumption}

As already mentioned, in this paper we focus on reference dynamics given by an SDE of the form 
\begin{equation}\label{eq:R-SDE-main-results}
\begin{split}
    dX^\eta_t &= f(t, X_t) dt + \sqrt{\eta}\sigma(t, X_t) dW_t, \ \ t \in [0, 1], \\ 
    X_0 &\sim \mu,
\end{split}
\end{equation}
whose path measure on $\calC_1 \coloneq C([0,1]: \bR^d)$ we denote by $R_\eta$. Similarly, the associated Schr\"odinger bridge is denoted by $\pi_\eta$.
The full set of regularity assumptions we impose on these dynamics are given in Section \ref{sec:development}, as Assumptions \ref{ass:b-sigma-local-properties}-\ref{ass:log-gradient-convergence}.
The first result, which gives us the necessary bridge from an LDP for the static problem to the dynamic setting, concerns bridge processes associated with the reference measure $R _\eta$.

\begin{theorem}[Uniform Laplace principle of bridges]
\label{thm:uniform-laplace-principle}
    Under Assumptions \ref{ass:b-sigma-local-properties}--\ref{ass:log-gradient-convergence}, for any compact set $K \Subset D$,
    \begin{equation}
        \lim_{\eta \downarrow 0} \sup_{x,y \in K} \left|-\eta \log \bE\left[\exp - \frac{1}{\eta} F(X^{\eta, xy}) \right] - \inf_{\varphi \in \calC_1} \left[F(\varphi) + I_B^{xy}(\varphi)\right]\right| = 0,
    \end{equation}
    where $X^{\eta, xy}$ is given by the dynamics \ref{eq:R-SDE-main-results}, conditioned on starting in $x$ at $t=0$ and ending in $y$ at $t=1$.
    The good rate function $I_B$ is given by
    \[ I_B(\varphi) = \inf_{\nu \in U^{xy}_\varphi} \frac{1}{2} \int_0^1 |\nu(t)|^2 dt,\]
    where $U^{xy}_\varphi \coloneq \{\nu \in L^2([0,1]:\bR^d): \varphi(\cdot) = \int_0^{\cdot} (b(s, \varphi(s) + \sigma(s, \varphi(s))\nu_s)ds\}$.
\end{theorem}
Combining this uniform Laplace principle with existing LDPs for the static problem, can obtain an LDP for the dynamic problem. Such LDPs for the static problem are studied in \cite{nilsson2025large} and typically require additional assumptions, most importantly the following.
\begin{assumption}
    The sequence of cost functions $c_\eta \coloneq -\eta \log \left(\frac{dR_\eta}{d(\mu \times \nu)}\right)$ converges uniformly on compacts in $D \times D$ to a function $c: D \times D \to \bR$, as $\eta \downarrow 0$.
\end{assumption}
Under this additional assumption we have the static LDP, with a rate function $I_S(x, y)$ that is equal to \[(c - (-\psi \oplus \phi))(x,y) = c(x, y) - (-\psi(x) + \phi(y)),\] under the assumption of compact supports.
Here $(\psi, \phi)$ are Kantorovich potentials \cite[Ch. 5]{Vil09} related to the Kantorovich problem for the limit cost function $c$.
\begin{theorem}
\label{thm:LDP_dSB}
    The sequence of Schr\"odinger bridges $\{\pi_\eta\}_{\eta > 0}$ satisfies an LDP on $\calC_1$ with the rate function $I_D$ given by
    \begin{equation}
    \begin{split}
        I_D(\varphi) &= I_S(\varphi_0, \varphi_1) + I_B^{\varphi_0 \varphi_1}(\varphi) \\
        &= (c - (-\psi \oplus \phi))(\varphi_0, \varphi_1) + \inf_{\nu \in U^{\varphi_0 \varphi_1}_\varphi} \frac{1}{2} \int_0^1 |\nu(t)|^2 dt,
    \end{split}
    \end{equation}
    where $U^{xy}_\varphi \coloneq \{\nu \in L^2([0,1]:\bR^d): \varphi(\cdot) = \int_0^{\cdot} (b(s, \varphi(s) + \sigma(s, \varphi(s))\nu_s)ds\}$, and $(\psi, \phi)$ are a pair Kantorovich potentials.
\end{theorem}

We note in particular that, under the assumption of compact supports, the special case that is our result applied to standard Brownian motion reference dynamics, i.e., taking $f \equiv 0$ and $\sigma \equiv I_d$, the $d$-dimensional identity matrix, is precisely the main result of \cite{Kato24}.

\begin{corollary}
\label{cor:BM}
    Let $R _\eta$ be the path measure associated with an $\sqrt{\eta}$-scaled Brownian motion. Then, the associated Schr\"odinger bridges satisfy the LDP of Theorem \ref{thm:LDP_dSB} with rate function given by %
    \begin{equation}
    \begin{split}
        I_D(\varphi) = -(-\psi \oplus \phi)(\varphi_0, \varphi_1) + \frac{1}{2} \int_0^1 |\varphi'_t|^2 dt. %
    \end{split}
    \end{equation}
    where $(\psi, \phi)$ is a pair of Kantorovich potentials for the OT problem with quadratic cost $c(x, y) = \frac{1}{2}|x-y|^2$; the second term is interpreted as $+\infty$ for any non-absolutely continuous $\varphi \in \calC_1$.
\end{corollary}

\section{Preliminaries}
\label{sec:prelim}
As mentioned in Section \ref{sec:main}, we consider reference measures $R \in \calP(\calC_1)$, that are the path measures of the solution to an SDE started in $\mu$:
\begin{equation}\label{eq:R-SDE}
\begin{split}
    dX_t &= f(t, X_t) dt + \sigma(t, X_t) dW_t, \ \ t \in [0, 1], \\ 
    X_0 &\sim \mu,
\end{split}
\end{equation}
With $f=0$ and $\sigma(t, x) = \sqrt{\eta} \in \bR$, $R = R_\eta$ is just an ($\eta$-scaled) Brownian motion started in $\mu$, and we recover the setting used in \cite{Kato24} (see Corollary \ref{cor:BM}).

The relative entropy in the corresponding dynamic SB problem \ref{eq:dynamic-SB-def} can be decomposed, for any $\pi \in \mathcal{P}(\calC _1)$ such that $\pi \ll R$, as 
\begin{equation*}
\begin{split}
    \calH (\pi \mid\mid R) &= \bE ^{\pi} \left[ \log \frac{d \pi _{01} }{d R_{01}} (X_0, X_1) \right] + \bE ^{\pi} \left[ \log \frac{d \pi^{X_0 X_1} }{d R^{X_0 X_1}} (X) \right] \\
    &= \calH (\pi _{01} \mid\mid  R_{01}) + \int _{(\bR ^d)^2} \calH (\pi^{xy} \mid\mid R^{xy}) \pi _{01} (dx, dy),
\end{split}
\end{equation*}
where $\pi ^{xy}$ and $R ^{xy}$ denote the dynamics obtained from $\pi$ and $R$, respectively, when conditioning on the initial and final points $x$ and $y$. In cases where $\frac{d\pi}{dR}(X)$ does not exist, we have $\calH(\pi \mid \mid R) = \infty$. 
For $\pi _{01}$-a.s.\ every $(x,y) \in (\bR ^d)^2$, $\calH (\pi^{xy} \mid\mid R^{xy})$ is non-negative and zero if and only if $\pi^{xy} = R^{xy}$. 
Thus, the dynamic SB problem \ref{eq:dynamic-SB-def} amounts to the following problem:
\begin{equation}\label{eq:static-SB-from-dynamic}
    \hat{\pi}_{01} = \argmin_{\pi_{01} \in \Pi(\mu, \nu)} \calH (\pi_{01} \mid\mid R_{01}),
\end{equation}
combined with interpolating this coupling by $R^{\cdot}$, the bridge processes associated with $R$, as $\hat{\pi} = \hat{\pi}_{01} \otimes R^{\cdot}$.
Here, $\Pi(\mu, \nu)$ is the ordinary set of couplings appearing in optimal transport,
\begin{equation*}
    \Pi (\mu, \nu) \coloneq \left\{ \pi \in \calP (\bR^d \times \bR^d): \pi_0 = \mu, \ \pi_1 = \nu \right\}.
\end{equation*}

If in \eqref{eq:R-SDE} we replace one replaces $\sigma(t,x)$ with a diffusion-coefficient $\sqrt{\eta} \sigma(t,x)$, decaying with the noise-level parameter $\eta > 0$, we have a reference $R_\eta$ for each $\eta$, as for the $\eta$-scaled Brownian motion used in \cite{Kato24}. Thereby, keeping $\mu$ and $\nu$ fixed, we obtain a \emph{sequence} $\{\pi_\eta\}$ of Schr\"odinger bridges.
As $\eta \downarrow 0$, the plans $\pi _\eta$ behave more and more like optimal transport (OT) plans, i.e., more and more like a solution to the Monge-Kantorovich problem (for the static plans $\{\pi_{\eta, 01}\}_{\eta > 0}$),
\begin{equation}\label{eq:Monge-Kantorovich}
    \inf_{\pi \in \Pi(\mu, \nu)} \int c \, d\pi,
\end{equation}
for some cost function $c: \bR^d \times \bR^d \to \bR$.
It is a classical fact that, under a general large deviation (see Section \ref{sec:LDPs}) assumption on the reference sequence $\{R_\eta\}_{\eta > 0}$ (e.g., given by Schilder's or Freidlin-Wentzell's theorem), $\pi_\eta$ goes weakly to a solution $\pi$ of \eqref{eq:Monge-Kantorovich}, where $c$ is given by the rate function, a fact that also holds for the dynamic setting \cite{leonard2012schrodinger, mikami2004monge}.
This is however not itself an LDP, and no information about the speed of convergence is given.

\subsection{Large deviation principles}\label{sec:LDPs}

We say that a sequence of probability measures $\{\mu_\kappa\}_{\kappa \in (0, 1)}$ on a topological (Polish for our current purposes) space $S$ with its Borel $\sigma$-algebra $\calB$ satisfies a \emph{large deviation principle} (LDP) with the \emph{rate function} $I: S \to \bR$ if to any measurable set $A \subseteq S$, the following inequalities hold:
\begin{equation}\label{eq:LDP-def}
    -\inf_{x \in \interior{A}} I(x) \leq \liminf_{\eta \downarrow 0} \eta \log \mu_\eta(A) \leq \limsup_{\eta \downarrow 0} \eta \log \mu_\eta(A) \leq -\inf_{x \in \close{A}} I(x),
\end{equation}
where $\interior{A}$/$\close{A}$ denote the interior/closure of $A$, respectively.
In particular, the left- and right-most inequalities hold for any measurable open and closed set, respectively.
$I$ is said to be a \emph{good} rate function if its level sets, \[I^{-1}([0,a]), a \geq 0,\] are compact.

Under goodness of the rate function (on, say, a Polish space), an equivalent formulation \cite{BD19, Dembo98} of the large deviation principle is as a \emph{Laplace principle}: for any bounded continuous function $F: S \to \bR$, %
\begin{equation*}
    \lim_{\kappa \downarrow 0} -\eta \log \bE^{\mu_\eta} \left[ \exp \left\{-\frac{1}{\eta} F \right\} \right] = \inf_{x \in S} \left[F(x) + I(x)\right].
\end{equation*}

Intuitively, an LDP characterizes in an exact way the asymptotic exponential escape of probability mass away from events.
In a very rough sense, these definitions say that the probability of an event $A$ decays like $e^{-\frac{1}{\kappa} \inf_{x \in A} I(x)}$ as $\eta \downarrow 0$.

It is shown in \cite{bernton2022entropic} that the static SB plans follow a ``weak-type'' (in the terminology of \cite{Kato24}) LDP. 
If one assumes that the marginals $\mu$ and $\nu$ have compact supports, this ``weak-type'' LDP is further a (\emph{full}) LDP.
In \cite{Kato24}, the author shows that the dynamic SBs also follow a weak-type LDP, for the specific setup with Brownian reference dynamics.
The main tool applied in \cite{Kato24} is the fact that the bridge processes reference dynamics follow an LDP with exponential continuity, and the proof relies heavily on the Gaussian properties of the reference process. 

In an effort to extend the results of \cite{Kato24} to other types of dynamics $R_\eta$, for example given by an SDE like \eqref{eq:R-SDE}, we here work directly with properties of conditioned SDEs, and arrive at our sought-after uniform Laplace principle by adapting the methods in \cite{BD19}. We then combine this uniform Laplace principle with the Laplace principle for static SBs with compact supports to obtain an LDP for dynamic SBs.

\section{A uniform Laplace principle for bridge processes}
\label{sec:development}

Henceforth we assume that the reference path measures $\{R_\eta\}$ correspond to SDEs of the form
\begin{equation}\label{eq:R-eta-SDE}
    dX^{\eta}_t = b(t, X^{\eta}_t) dt + \sqrt{\eta} \sigma(t, X^{\eta}_t) dW_t, \ \ X^{\eta}_0 \sim \mu.
\end{equation}
where $b: [0, 1] \times \close{D} \to \bR^d$ and $\sigma: b: [0, 1] \times \close{D} \to \bR^{d \times d}$ are defined on an (open, connected) domain $D \subseteq \bR^d$.
\begin{assumption}\label{ass:b-sigma-local-properties}
    $b$ and $\sigma$ are locally bounded, $b$ is locally Lipschitz, and $\sigma$ is $C^1$.
    \eqref{eq:R-eta-SDE} has a strong solution, which stays in $\close{D}$ with probability $1$.
\end{assumption}%
\noindent The local Lipschitz conditions imply strong uniqueness of the solution on $[0, 1]$.

The law of the bridge processes of the related dynamic SBs $\{\pi_\eta\}$, conditioned on its endpoints, is the same as that of \eqref{eq:R-eta-SDE}, by the discussion in Section \ref{sec:prelim}.
The next assumption is aimed at ensuring well-behaved transition densities for these bridge processes.
\begin{assumption}\label{ass:transition-density-exists-and-Hloc}
    The process $X^\eta$ admits a smooth transition density $p(t, X^{\eta, xy}_t; 1, y)$.
    For $t=0$, $p(0, x; \cdot, \cdot)$ is in the space $L^2([t_0,1]; \tilde{H}_{\mathrm{loc}})$, as defined in \cite{haussmann1986time}, for any $t_0 > 0$. %
\end{assumption}

Assumption \ref{ass:transition-density-exists-and-Hloc} follows, for example, from the following assumptions used in \cite{menozzi2021density}.
\begin{taggedassumption}{\ref*{ass:transition-density-exists-and-Hloc}'}
    The drift $b$ satisfies a linear growth property. For $\sigma$, there exists $\kappa_0 > 0$ such that for any $t \in [0,1], x \in \close{D}, v \in \bR^d$, $\kappa_0^{-1} |v|^2 \leq \langle v, \sigma \sigma^\top(t,x) v \rangle \leq \kappa_0 |v|^2$.
    Further, there is a constant $\alpha \in (0,1)$ such that $\norm{\sigma(t,x) - \sigma(t,y)} \leq \kappa_0 |x - y|^\alpha$.
\end{taggedassumption}
Under Assumption \ref{ass:transition-density-exists-and-Hloc}', the results of \cite{menozzi2021density} ensure that the transition density $p(0, x; t, \cdot)$ exists, is smooth and sub-Gaussian, and has a sub-Gaussian gradient.
The scale of this sub-Gaussian grows with $t$, bounding $p(0, x; t, \cdot)$ and its derivatives for $t \geq t_0$, for any choice of $t_0 > 0$.
This means that $p(0, x; \cdot, \cdot)$ is in the space $L^2([t_0,1]; \tilde{H}_{\mathrm{loc}})$ for any $t_0 > 0$.

The law of the bridges of \eqref{eq:R-eta-SDE}, obtained by conditioning on the initial and final point of $X^\eta$, is given by the corresponding Doob $h$-transform \cite{baudoin2002conditioned, rogers2000diffusions}. Denoting by $\{X^{\eta, xy}_t\}$ the bridge process which is $x$ at $t=0$ and $y$ at $t=1$, it follows that
\begin{equation}\label{eq:R-eta-SDE-h-transform}
\begin{split}
    dX^{\eta, xy}_t = &b(t, X^{\eta, xy}_t) dt + \sqrt{\eta} \sigma(t, X^{\eta, xy}_t) dW_t \\
    &+ g ^y _\eta(t, X^{\eta, xy}_t) dt, \ \ X^{\eta, xy}_0 = x,
\end{split}
\end{equation}
on $[0, 1)$, where $g _\eta ^{y} (t,x)$ is the $h$-transform term, defined by
\begin{align}
\label{eq:hTrans}
g ^y _\eta (t,x) = \eta \sigma \sigma^\top(t, x) \nabla_x \log p(t, x; 1, y).
\end{align}
If $X^\eta$ is pure Brownian motion (i.e., $b = 0$ and $\sigma = I_d$ in \eqref{eq:R-eta-SDE}), $g_\eta^y(t, x) = -\frac{x - y}{t}$ for all $\eta > 0$.

By Assumption \ref{ass:transition-density-exists-and-Hloc}, the process $X^{\eta}$, and specifically $X^{\eta, x}$ for any $x \in D$, can be reversed in time. %
Let us use the notation $\hat{X}_t \coloneq X_{1-t}$ for a general process $X$.
We can then, as an alternative to \eqref{eq:R-eta-SDE-h-transform}, use the processes $\hat{X}^{\eta, x}_t = X^{\eta, x}_{1-t}$ to condition on $x$ in the opposite (time) direction.
The marginal density of $X^{\eta, x}$ is $p(0, x; t, z)$, and thus its time reversal is 
\begin{equation}\label{eq:reverse-x}
\begin{split}
    d\hat{X}^{\eta, x}_t = &-b(1 - t, \hat{X}^{\eta, x}_t) dt + \sqrt{\eta} \sigma(1 - t, \hat{X}^{\eta, x}_t) d\hat{W}_t \\
    &+ \overset{\leftarrow}{g}{}^x_\eta(t, \hat{X}^{\eta, x}_t) dt,%
    \ \ \hat{X}^{\eta, x}_0 = X^{\eta, x}_1,
\end{split}
\end{equation}
where the last drift term is defined by
\begin{equation}\label{eq:g-backwards-def}
    [\overset{\leftarrow}{g}{}^x_\eta(t, \hat{x})]_i \coloneq \eta p(0, x; 1-t, \hat{x})^{-1} \sum_{j=1}^d \frac{\partial}{\partial \hat{x}_j} [\sigma \sigma^\top_{ij}(1 - t, \hat{x}) p(0, x; 1-t, \hat{x})].
\end{equation}
Conditioning $X^{\eta, x}$ on $X_1 = \hat{X}_0 = y$ yields $X^{\eta, xy}_{\cdot} = \hat{X}^{\eta, xy}_{1-\cdot}$, following \eqref{eq:reverse-x} started in $y$:
\begin{equation}\label{eq:reverse-xy}
\begin{split}
    d\hat{X}^{\eta, xy}_t = &-b(1 - t, \hat{X}^{\eta, xy}_t) dt + \sqrt{\eta} \sigma(1 - t, \hat{X}^{\eta, xy}_t) d\hat{W}_t \\
    &+ \overset{\leftarrow}{g}{}^x_\eta(t, \hat{X}^{\eta, xy}_t) dt,%
    \ \ \hat{X}^{\eta, xy}_0 = y.
\end{split}
\end{equation}

The bridges \eqref{eq:R-eta-SDE-h-transform} inherit the strong existence and uniqueness properties on $[0, 1)$ from $X^{\eta, x}$. %
However, it is not possible to put a uniform (local) Lipschitz assumption in $[0,1)$ on the drift term $g^y_\eta$, as done on $b$, since this would even rule out Brownian bridges.
This is perhaps the main complication vis-\`a-vis the analysis in \cite[Section 3.2]{BD19}, and why the results do not fall under the classical Freidlin-Wentzell theory.

The main aim of this section is to show the uniform Laplace principle in Theorem \ref{thm:uniform-laplace-principle}. We start by proving the following Laplace principle for the bridge processes $X^{\eta, xy}$.
\begin{theorem}
\label{thm:LDP_bridge}
Assume Assumptions \ref{ass:b-sigma-local-properties}-\ref{ass:log-gradient-convergence} hold and fix $x,y$. For $F$ bounded and continuous, the sequence $\{ X^{\eta, xy} \} _\eta $ of bridge processes satisfies 
\begin{align*}
    \lim _{\eta \to 0} -\eta \log \mathbb{E} \left[ \exp\left\{ -\frac{1}{\eta} F(X^{\eta, xy}) \right\} \right] = \inf _{\varphi \in \calC _1} \left[ F(\varphi) + \inf _{\calU _\varphi} \frac{1}{2} \int _0 ^1 | \nu _s | ^2 ds  \right].
\end{align*}
\end{theorem}
As starting point for the analysis we use the representation formula given in Proposition \ref{prop:representation1}, or more specifically the finite-energy version in Proposition \ref{prop:BBrepM}. Because of the assumptions placed on the functions involved, both representations for the bridge processes follow from the corresponding representations for the unconditioned SDEs (see Theorems 3.14 and 3.17, respectively, in \cite{BD19}) together with an application of the $h$-transform.

Let ${\calG_t}$ denote the completion of the filtration generated by $\{W_t\}$ and let $\calA$ denote the collection of admissible controls:
\begin{align}\label{eq:defA}
    \calA = \left\{ \{\nu _t \}_{t \in [0,1]}: \ \nu \textrm{ is } \calG _t\textrm{-progressively measurable}, \ \bE \left[ \int _0 ^1 |\nu _t| ^2 dt \right] < \infty \right\}.
\end{align}

\begin{proposition}[Variational representation for bridges]\label{prop:representation1}
For any bounded, continuous function $F: C([0,1]:\bR ^d) \to \bR$, we have
\begin{equation}\label{eq:representation}
    -\eta \log \bE \left[ \exp \left\{-\frac{1}{\eta} F(X^{\eta, xy}_t)\right\} \right] = \inf_{\nu^\eta \in \mathcal{A}} \bE \left[F(\bar{X}^{\eta, xy}_t) + \frac{1}{2}\int_0^1 | \nu^\eta_s |^2 ds \right],
\end{equation}
where $\calA$ is given in \eqref{eq:defA} and $\bar{X}^{\eta, xy} = \{ \bar{X}^{\eta, xy}_t \}_{t \in [0, 1]} $ is the controlled Brownian bridge characterized by (again interpreted componentwise)
\begin{equation}\label{eq:controlled-BB}
    d\bar{X}^{\eta, xy}_t = (b + g_\eta^y)(t, \bar{X}^{\eta, xy}_t) dt + \sqrt{\eta} \sigma dW_t + \sigma \nu^\eta_t dt, \ \ \bar{X}^{\eta, xy}_0 = x,
\end{equation}

\end{proposition}

It turns out that in the representation \eqref{eq:representation}, for a near-optimal choice of control, it is enough to consider controls that lie in a compact set; the following is \cite[Theorem 3.17]{BD19}, stated for our bridge processes. %
\begin{proposition}
\label{prop:BBrepM}
Take a bounded, continuous function $F: \calC_1 \to \bR$ and let $\delta>0$. Then there exists an $M < \infty$, depending on $\norm{F} _\infty$ and $\delta$, such that for all $\eta \in (0,1)$,
\begin{equation}\label{eq:representationM}
    -\eta \log \bE \left[ \exp \left\{-\frac{1}{\eta} F(X^{\eta, xy})\right\} \right] \geq \inf_{\nu^\eta \in \mathcal{A}_{b,M}} \bE \left[F(\bar{X}^{\eta, xy}_t) + \frac{1}{2}\int_0^1 | \nu^\eta_s |^2 ds \right] - \delta,
\end{equation}
where $\mathcal{A}_{b,M} \coloneq \{\nu \in \calA : \nu(\omega) \in S_M \ \forall \omega\}$ and $S_M \coloneq \{\phi \in \calL^2([0,1]:\bR^d) : \int_0^1 |\phi(t)|^2 dt \leq M \}$.
\end{proposition}

Similar to the case of unconditioned small-noise diffusions, Proposition \ref{prop:BBrepM} plays a key role in establishing the Laplace principle for our bridge processes. 
The key point is that it allows us to work with sequences of pairs of controlled processes and controls $\{(\bar{X}^{\eta, xy}_t, \nu^\eta)\}_{\eta \in (0,1)}$ that are \emph{tight}. %
This is established in Section \ref{sec:tightness-of-controlled-via-Lipschitz} by proving that both marginal sequences, $\{\bar{X}^{\eta, xy}\}_{\eta}$ and $\{\nu^\eta\}_{\eta}$, are tight. The second marginal, $\{\nu^\eta\}_{\eta} \subseteq S_M$, is tight in the \emph{weak topology} of $\calL^2([0,1]:\bR^d)$, since, by the Banach-Alaoglu theorem, $S_M$ is compact.
Because $g_\eta^y$ is typically not uniformly Lipschitz, tightness of the first marginal is more difficult to show in the case considered here. 
To achieve this, we add a few regularity assumptions regarding the bridge dynamics. %
Next, in Section \ref{sec:convergence-to-controlled-ode} we characterize the limit points of any convergent (sub)sequences of $\{(\bar{X}^{\eta, xy}_t, \nu^\eta)\}_{\eta \in (0,1)}$ in terms of a limit ODE. We then prove Theorem \ref{thm:LDP_bridge} in Section \ref{sec:LaplaceBridge} by proving the corresponding upper and lower Laplace bounds. Lastly, we obtain the uniform version of the Laplace princple, stated in Theorem \ref{thm:uniform-laplace-principle} and necessary for the extension of the LDP to dynamic Schr\"odinger bridges, in Section \ref{sec:uniform}.

\subsection{Tightness of controlled processes}\label{sec:tightness-of-controlled-via-Lipschitz}

To show the LDP with the weak convergence approach, we first need to verify that the sequence of controlled bridge processes $\{\bar{X}^{\eta, xy}\}_\eta$ is tight in $\calC_1$ (under $\bP$), for any $M \geq 0$, and any sequence of controls $\{\nu^\eta\}_{\eta} \subseteq \calA_{b, M}$. 
To achieve this we impose the following additional assumptions.

\begin{assumption}\label{ass:bounded-lipshitzness-of-log-gradients}
    For each $K \Subset D$ and $\delta > 0$, there exists a constant $L_{K, \delta}$ such that for $t - s \geq \delta, x, y \in K$, $\eta \sigma\sigma^\top(s, x) \nabla_x \log p_\eta(s, x; t, y)$ and %
    $\overset{\leftarrow}{g}{}^x_\eta(t, \hat{x})$, defined in \eqref{eq:g-backwards-def}, are Lipschitz in $x$ and $y$ respectively, with Lipschitz constant $L_{K, \delta}$. %
\end{assumption}

\begin{assumption}\label{ass:exp-boundedness}
    For each $K \Subset D$ and $M > 0$, there exists a $K_M \Subset D$ such that for any $(x,y) \in K \times K$, it holds that \[\limsup_{\eta \downarrow 0} \sup_{(x,y) \in K \times K} \eta \log \bP(\tau^{\eta, xy}_{K_M} \neq \infty) \leq -M,\] %
    where $\tau^{\eta, xy}_{K_M} \coloneq \inf \{t \in [0, 1]: X^{\eta, xy}_t \notin K_M\}$. %
\end{assumption}
One way to obtain Assumption \ref{ass:exp-boundedness} is when the underlying unconditioned process satisfy an LDP of Freidlin-Wentzell type combined with the additional regularity assumptions. Conditioning on the start and end points within a compact set then preserves the exponential tightness on $[0, 1-\delta]$, for any $\delta > 0$, and for $[1-\delta, 1]$ we can use time reversal. Alternatively, we can obtain Assumption \ref{ass:exp-boundedness} from assuming a Lyapunov-type condition for the unconditioned process. 

The next assumption may seem vacuous, and is indeed here mostly for presentational simplicity.

\begin{assumption}
    For all $x, y \in D$ we have that the $h$-transform backward and forward versions both satisfy $\bP(X^{\eta, xy}_0 = x) = \bP(X^{\eta, xy}_1 = y) = 1$.
\end{assumption}

\begin{assumption}\label{ass:log-gradient-convergence}
    The function $(t,x,y) \mapsto g^y_\eta(t, x) = \eta \sigma \sigma^\top \nabla_x \log p_\eta(s, x; 1, y)$ converges locally uniformly on compacts  in $[0, 1) \times K^2$, for every compact $K$, to a function $g^y(s, x)$ as $\eta \downarrow 0$.%
\end{assumption}

Consider the controlled bridge sequence $\{\bar{X}^{\eta, xy}_t\}$, where the controls are limited to $S_M$. 
Under the above assumptions, we get that, for all $\eta \in (0,1)$, and $x, y \in K$, using the Cauchy-Schwarz inequality like before,
\begin{equation}\label{eq:X_eta_xy_controlled_Cauchy_Schwarz}
\begin{split}
    &\bP\left(\sup_{|t-s| \leq \delta} |\bar{X}^{\eta, xy}_t - \bar{X}^{\eta, xy}_s| \geq \xi\right) \\
    &\qquad \leq \bE^{\bQ^\eta} \left[\frac{d\bP}{d\bQ^\eta} \mathbbm{1}\{\sup_{|t-s| \leq \delta} |\bar{X}^{\eta, xy}_t - \bar{X}^{\eta, xy}_s| \geq \xi\}\right] \\
    &\qquad \leq \bE^{\bQ^\eta}\left[\left(\frac{d\bP}{d\bQ^\eta}\right)^2\right]^{\frac{1}{2}} \bQ^\eta\left(\sup_{|t-s| \leq \delta} |\bar{X}^{\eta, xy}_t - \bar{X}^{\eta, xy}_s| \geq \xi\right) \\
    &\qquad \leq e^{\frac{M}{\eta}} \bP\left(\sup_{|t-s| \leq \delta} |X^{\eta, xy}_t - X^{\eta, xy}_s| \geq \xi\right).
\end{split}
\end{equation}
Now, based on $K$, take $K_{M + 1}$ according to Assumption \ref{ass:exp-boundedness}.
Then, 
\begin{equation}\label{eq:X-eta-xy-moc-split-exponential-events}
\begin{split}
    &\bP\left(\sup_{|t-s| \leq \delta} |X^{\eta, xy}_t - X^{\eta, xy}_s| \geq \xi\right) \\
    &\qquad \leq \bP\left(\tau^{\eta, xy}_{K_{M+1}} \neq \infty \right) + \bP\left(\sup_{|t-s| \leq \delta} |X^{\eta, xy}_t - X^{\eta, xy}_s| \geq \xi, \tau^{\eta, xy}_{K_{M+1}} = \infty \right) \\
    &\qquad \leq e^{-\frac{M + 1}{\eta}} + \bP\left(\sup_{|t-s| \leq \delta} |X^{\eta, xy}_t - X^{\eta, xy}_s| \geq \xi, \tau^{\eta, xy}_{K_{M+1}} = \infty \right).
\end{split}
\end{equation}
First consider the process $X^{\eta, xy}$ restricted to $t \leq \frac{2}{3}$.
For all $\omega \in \{\tau^{\eta, xy}_{K_{M+1}} = \infty\}$, we have that $b(t, X^{\eta, xy}_t)$ and $\eta\sigma\sigma^\top \nabla_x \log p_\eta(t, X^{\eta, xy}_t; 1, y)$ are bounded by Assumption \ref{ass:bounded-lipshitzness-of-log-gradients}, so
\begin{equation*}
    \int_0^{\cdot} (b(u, X^{\eta, xy}_u) + \eta \sigma\sigma^\top \nabla_x \log p_\eta(u, X^{\eta, xy}_u; 1, y)) \, du
\end{equation*}
is uniformly absolutely continuous on $[0, \frac{2}{3}]$ over such $\omega$.
Thus, for small $\delta$, and all such $\omega$, 
\begin{equation*}
    \int_{s}^{t} (b(u, X^{\eta, xy}_u) + \eta \sigma\sigma^\top \nabla_x \log p_\eta(u, X^{\eta, xy}_u; 1, y)) \, du \leq \frac{\xi}{2}
\end{equation*}
whenever $t - s \leq \delta$. 
Thus, for small enough $\delta$,
\begin{equation}\label{eq:bound-move-tau-into-itointegral}
\begin{split}
    \bP&\left(\sup_{\substack{s,t \in [0, 2/3] \\ |s-t| \leq \delta}} |X^{\eta, xy}_t - X^{\eta, xy}_s| \geq \xi, \tau^{\eta, xy}_{K_{M+1}} = \infty \right) \\
    &\leq \bP\left(\sup_{\substack{s,t \in [0, 2/3] \\ |s-t| \leq \delta}} \left|\int_{s}^{t} \sqrt{\eta} \sigma(u, X^{\eta, xy}_u) dW_u\right| \geq \frac{\xi}{2}, \tau^{\eta, xy}_{K_{M+1}} = \infty \right) \\
    &= \bP\left(\sup_{\substack{s,t \in [0, 2/3] \\ |s-t| \leq \delta}} \left|\int_{s}^{t} \sqrt{\eta} \sigma(u, X^{\eta, xy}_{u \land \tau^{\eta, xy}_{K_{M+1}}}) dW_u\right| \geq \frac{\xi}{2}, \tau^{\eta, xy}_{K_{M+1}} = \infty \right) \\
    &\leq \bP\left(\sup_{\substack{s,t \in [0, 2/3] \\ |s-t| \leq \delta}} \left|\int_{s}^{t} \sqrt{\eta} \sigma(u, X^{\eta, xy}_{u \land \tau^{\eta, xy}_{K_{M+1}}}) dW_u\right| \geq \frac{\xi}{2} \right).
\end{split}
\end{equation}

The last expression in \eqref{eq:bound-move-tau-into-itointegral} may be further bounded by noting its bounded quadratic variation:
Since $X^{\eta, xy}_{t \land \tau^{\eta, xy}_{M+1}} \in K_{M+1}$ for all $t \in [0,1]$, its norm is bounded as $|X^{\eta, xy}_{t \land \tau^{\eta, xy}_{M+1}}| \leq \sup_{x \in K_{M+1}} |x|$ , and by the linear growth (with constant $L$) of $\sigma(t,x)$, we get $\norm{\sigma(t, X^{\eta, xy}_{t \land \tau^{\eta, xy}_{M+1}})} \leq L (1 + \sup_{x \in K_{M+1}} |x|)$.
This may be leveraged in the following classical estimate for a one-dimensional continuous local martingale $V$ started from the origin \cite{dung2023some, chafai2020subgaussian_blog}:
\begin{equation}\label{eq:classical-estimate}
    \bP(\sup_{s \in [0, t]} V_s \geq \xi, \langle V \rangle_t \leq c) \leq e^{-\frac{\xi^2}{2 c}},
\end{equation}
If $V$ is instead $d$-dimensional and $\langle V^{(i)} \rangle_t \leq c, 1 \leq i \leq d$ holds a.s., \eqref{eq:classical-estimate} takes the following version by a union bound:
\begin{equation}\label{eq:classical-estimate-d}
    \bP(\sup_{s \in [0, t]} |V_s| \geq \xi) \leq 2d e^{-\frac{(\xi/\sqrt{d})^2}{2 c}}.
\end{equation}

To apply the quadratic variation bound and \eqref{eq:classical-estimate-d} to \eqref{eq:mesh}, consider a sequence of meshes $\{t^{(\delta)}_i\}_{i=0}^{N^{(\delta)}}$, one for each $\delta$, with $t^{(\delta)}_0 = 0$, $t^{(\delta)}_{N^{(\delta)}} = 2/3$, and $t^{(\delta)}_i < t^{(\delta)}_{i+1}$ for all $i$, such that $\delta \leq t^{(\delta)}_{i+1} - t^{(\delta)}_i \leq 2 \delta$.
Then, since any subinterval $[s, t] \subseteq [0, 2/3]$ shorter than $\delta$ intersects at most two intervals from the mesh, we have
\begin{equation}\label{eq:mesh}
\begin{split}
    \bP&\left(\sup_{\substack{s,t \in [0, 2/3] \\ |s-t| \leq \delta}} \left|\int_{s}^{t} \sqrt{\eta} \sigma(u, X^{\eta, xy}_{u \land \tau^{\eta, xy}_{K_{M+1}}}) dW_u\right| \geq \frac{\xi}{2} \right) \\
    &\leq \bP\left(\bigcup_{i=0}^{N^\delta - 1} \sup_{\substack{s,t \in [t^{(\delta)}_i, t^{(\delta)}_{i+1}]}} \left|\int_{s}^{t} \sqrt{\eta} \sigma(u, X^{\eta, xy}_{u \land \tau^{\eta, xy}_{K_{M+1}}}) dW_u\right| \geq \frac{\xi}{4} \right) \\
    &\leq \sum_{i=0}^{N^\delta - 1} \bP\left(\sup_{\substack{s,t \in [t^{(\delta)}_i, t^{(\delta)}_{i+1}] \\}} \left|\int_{s}^{t} \sqrt{\eta} \sigma(u, X^{\eta, xy}_{u \land \tau^{\eta, xy}_{K_{M+1}}}) dW_u\right| \geq \frac{\xi}{4} \right) \\
    &\leq \sum_{i=0}^{N^\delta - 1} \bP\left(\sup_{\substack{t \in [t^{(\delta)}_i, t^{(\delta)}_{i+1}]}} \left|\int_{t^{(\delta)}_i}^{t} \sqrt{\eta} \sigma(u, X^{\eta, xy}_{u \land \tau^{\eta, xy}_{K_{M+1}}}) dW_u\right| \geq \frac{\xi}{8} \right)
\end{split}
\end{equation}
Now, we can bound the quadratic variation on $[t^{(\delta)}_i, t^{(\delta)}_{i+1}]$ for each dimension $i$:
\begin{equation*}
\begin{split}
    \langle\int_{t^{(\delta)}_i}^{\cdot} \sqrt{\eta} \sigma_{i}(u, X^{\eta, xy}_{u \land \tau^{\eta, xy}_{K_{M+1}}}) dW_u\rangle_t &= \int_{t^{(\delta)}_i}^{t} \sum_{j=1}^d |\sqrt{\eta} \sigma_{ij}(u, X^{\eta, xy}_{u \land \tau^{\eta, xy}_{K_{M+1}}})|^2 du \\
    &\leq \eta L^2 (1 + \sup_{x \in K_{M+1}} |x|)^2 (t-t^{(\delta)}_i) \\
    &\leq \eta L^2 (1 + \sup_{x \in K_{M+1}} |x|)^2 2 \delta.
\end{split}
\end{equation*}
Thus, using \eqref{eq:classical-estimate-d}, the final expression in \eqref{eq:mesh} is bounded by
\begin{equation*}
\begin{split}
    &N^\delta \times 2 \exp\left({-\frac{(\frac{\xi / \sqrt{d}}{8})^2}{2 \eta L^2 (1 + \sup_{x \in K_{M+1}} |x|)^2 2 \delta}}\right) \\ 
    &\leq 4\delta^{-1} \exp\left({-\frac{\xi^2}{256 \eta d L^2 (1 + \sup_{x \in K_{M+1}} |x|)^2 \delta}}\right).
\end{split}
\end{equation*}

Therefore, selecting $\delta_\xi$ small enough, for all $\delta \leq \delta_\xi$, it is bounded by $e^{-\frac{M + 1}{\eta}}$.
When combining with \eqref{eq:X-eta-xy-moc-split-exponential-events}, we get
\begin{equation*}
\begin{split}
    \bP\left(\sup_{\substack{s,t \in [0, 2/3] \\ |s-t| \leq \delta}} |X^{\eta, xy}_t - X^{\eta, xy}_s| \geq \xi\right) \leq e^{-\frac{M + 1}{\eta}} + e^{-\frac{M + 1}{\eta}} = 2e^{-\frac{M + 1}{\eta}}.
\end{split}
\end{equation*}
\eqref{eq:X_eta_xy_controlled_Cauchy_Schwarz} now yields
\begin{equation*}
\begin{split}
    \bP\left(\sup_{\substack{s,t \in [0, 2/3] \\ |s-t| \leq \delta}} |\bar{X}^{\eta, xy}_t - \bar{X}^{\eta, xy}_s| \geq \xi\right) 
    &\leq e^{\frac{M}{\eta}} \bP\left(\sup_{|t-s| \leq \delta} |X^{\eta, xy}_t - X^{\eta, xy}_s| \geq \xi\right) \\
    &\leq e^{\frac{M}{\eta}} 2e^{-\frac{M + 1}{\eta}} = 2e^{-\frac{1}{\eta}}.
\end{split}
\end{equation*}
Then, for any decreasing subsequence $\{\eta_i\}_{i=1}^\infty \subseteq (0,1)$, and any $\varepsilon > 0$, take $n_{\xi, \varepsilon} \in \bN$ such that $2e^{-\frac{1}{\eta_{n_{\xi, \varepsilon}}}} \leq \varepsilon$.
Then, for any $i \geq n_{\xi, \varepsilon}$, \[\bP\left(\sup_{\substack{s,t \in [0, 2/3] \\ |s-t| \leq \delta}} |\bar{X}^{\eta_i, xy}_t - \bar{X}^{\eta_i, xy}_s| \geq \xi\right) \leq \varepsilon.\]
Now, set $\tilde{\delta}_{\xi, \varepsilon}$ small enough so that 
for $i < n_{\xi, \varepsilon}$, \[\bP\left(\sup_{\substack{s,t \in [0, 2/3] \\ |s-t| \leq \delta}} |\bar{X}^{\eta_i, xy}_t - \bar{X}^{\eta_i, xy}_s| \geq \xi\right) \leq \varepsilon\] for all $\delta \leq \tilde{\delta}_{\xi, \varepsilon}$.
With $\delta_{\xi, \varepsilon} \coloneq \delta_{\xi} \land \tilde{\delta}_{\xi, \varepsilon}$, we have then that for all $i \in \bN$ and $\delta \leq \delta_{\xi, \varepsilon}$,
\begin{equation*}
    \sup_{i \in \bN}\bP\left(\sup_{\substack{s,t \in [0, 2/3] \\ |s-t| \leq \delta}} |\bar{X}^{\eta_i, xy}_t - \bar{X}^{\eta_i, xy}_s| \geq \xi\right) \leq \varepsilon.
\end{equation*}
Since also $\bP\left(\bar{X}^{\eta_i, xy}_0 = x\right) = 1$, $\{\bar{X}^{\eta_i, xy}\}_{i=1}^\infty$ is tight.
Since the subsequence is arbitrary, so is $\{\bar{X}^{\eta, xy}\}_{\eta \in (0,1)}$.

\begin{remark}
    The above argument was restricted to $[0, 2/3]$.
    By reversing the process (as per \eqref{eq:reverse-xy}) and applying the same exponential arguments on $[0, 2/3]$ for the reversed bridges processes $\hat{X}^{\eta, xy}$, we get that 
    \begin{equation*}
        \lim_{\delta \downarrow 0} \sup_{i \in \bN}\bP\left(\sup_{\substack{s,t \in [1/3, 1] \\ |s-t| \leq \delta}} |\bar{X}^{\eta_i, xy}_t - \bar{X}^{\eta_i, xy}_s| \geq \xi\right) = 0,
    \end{equation*}
    implying tightness on $[1/3, 1]$.
    Combining the two, e.g. with a union bound argument, gives tightness on $[0,1]$.
\end{remark}

\subsection{Convergence to controlled ODE}\label{sec:convergence-to-controlled-ode}

The tightness established in Section \ref{sec:tightness-of-controlled-via-Lipschitz} implies that $(\bar{X}^{\eta, xy}, \nu^\eta)$ converges along some certain subsequences. The next step in establishing the Laplace principle in \ref{thm:LDP_bridge} is to characterize the limit points of such convergent subsequences; this is the analogue of Lemma 3.21 in \cite{BD19}. 

\begin{lemma}
    Assume that a controlled sequence $(\bar{X}^{\eta, xy}, \nu^\eta)$ converges to $(\bar{X}^{xy}, \nu)$.
    Then this limit satisfies, almost surely, the following integral equation on $[0,1)$%
    \begin{equation}\label{eq:limit-ODE}
        \bar{X}^{xy}_t = x + \int_0^{t} (b(s, \bar{X}^{xy}_s) + g^y(s, \bar{X}^{xy}_s) + \sigma(s, \bar{X}^{xy}_s) \nu_s) ds.
    \end{equation}
\end{lemma}

\begin{proof}
    We want to see that, term by term, the right-hand side of the equation
    \begin{equation}\label{eq:small-noise-converging-terms}
    \begin{split}
        \bar{X}^{\eta, xy}_t - x = \int_0^{t} b(s, \bar{X}^{\eta, xy}_s) ds + \int_0^{t} \eta \sigma \sigma^\top(s, \bar{X}^{\eta, xy}_s) \nabla_x \log p_\eta(s, \bar{X}^{\eta, xy}_s; 1, y) ds \\+ \int_0^{t} \sigma(s, \bar{X}^{\eta, xy}_s) \nu^\eta_s ds + \int_0^{t} \sqrt{\eta} \sigma(s, \bar{X}^{\eta, xy}_s) dW_s,
    \end{split}
    \end{equation}
    converges to the corresponding integral term in \eqref{eq:limit-ODE} for each $t \in [0, 1)$.
    The first and third terms go to $\int_0^{t} b(s, \bar{X}^{xy}_s) ds$ and $\int_0^{t} \sigma(s, \bar{X}^{xy}_s) \nu_s ds$ respectively, by the same arguments as in \cite{BD19} localized to some compact $K_\alpha$ with $x, y \in K$, $\alpha > 0$, and $\limsup_{\eta \downarrow 0} \eta \log \bP(\tau^{\eta, xy}_{K_\alpha} \neq \infty) \leq -\alpha < 0$.
    For the second term, denoting $g_\eta^y(s, x) \coloneq \eta \sigma \sigma^\top(s, x) \nabla_x \log p_\eta(s, x; 1, y)$, we have for $\omega \in \{\tau^{\eta, xy}_{K_\alpha} = \infty\}$,
    \begin{equation*}
    \begin{split}
        \Big|\int_0^{t} &g_\eta^y(s, \bar{X}^{\eta, xy}_s) ds - \int_0^{t} g^y(s, \bar{X}^{xy}_s) ds\Big| \\
        &\leq \int_0^{t} \Big| g_\eta^y(s, \bar{X}^{\eta, xy}_s) ds - g^y(s, \bar{X}^{xy}_s) \Big| ds \\
        &\leq \int_0^{t} \Big| g_\eta^y(s, \bar{X}^{\eta, xy}_s) ds - g^y(s, \bar{X}^{\eta, xy}_s) \Big| ds + \int_0^{t} \Big| g^y(s, \bar{X}^{\eta, xy}_s) ds - g^y(s, \bar{X}^{xy}_s) \Big| ds \\
        &\leq \int_0^{t} |g_\eta^y - g^y|_{\infty, [0,t] \times K_\alpha} ds + \int_0^{t} L_{K_\alpha, 1 - t} \Big| \bar{X}^{\eta, xy}_s - \bar{X}^{xy}_s \Big| ds, \\
    \end{split}
    \end{equation*}
    which goes to $0$ as $\eta \downarrow 0$ for any $t < 1$.
    Further the probability of the complement, $\{\tau^{\eta, xy}_{K_\alpha} \neq \infty\}$ goes to $0$, and thus $\int_0^{t} g_\eta^y(s, \bar{X}^{\eta, xy}_s) ds \to \int_0^{t} g^y(s, \bar{X}^{xy}_s) ds$ weakly.
    For the stochastic integral term, consider the stopped process $\{\bar{X}^{\eta, xy}_{t \land \tau^{\eta, xy}_{K_\alpha}}\}$.
    By the same arguments as in Section \ref{sec:tightness-of-controlled-via-Lipschitz}, its It\^o integral agrees with that of $\{\bar{X}^{\eta, xy}_{t}\}$ on $\{\tau^{\eta, xy}_{K_\alpha} = \infty\}$.
    Using the Burkholder-Davis-Gundy inequality and the linear growth property of $\sigma$, we have for some constant $C$:
    \begin{equation*}
        \bE \sup_{t \in [0,1]} \left|\sqrt{\eta}\int_0^t \sigma(s, \bar{X}^{\eta, xy}_{s \land \tau^{\eta, xy}_{K_\alpha}}) dW_s\right|^2 \leq C \eta \int_0^1 \bE [1 + |\bar{X}^{\eta, xy}_{s \land \tau^{\eta, xy}_{K_\alpha}}|^2] ds.
    \end{equation*}
    By the compactness of $K_\alpha$, this is bounded and goes to $0$ in the limit as $\eta \to 0$. %
    As before, the probability of the complement $\{\tau^{\eta, xy}_{K_\alpha} \neq \infty\}$ also goes to $0$, so $\sqrt{\eta} \int_0^{\cdot} \sigma(s, \bar{X}^{\eta, xy}_{s}) dW_s \to 0$ weakly in $\calC_1$.
    
\end{proof}

\subsection{Laplace principle for $\{X ^{\eta, xy}\} _{\eta}$}
\label{sec:LaplaceBridge}

With the tightness and convergence behavior in place, we can proceed to prove Theorem \ref{thm:LDP_bridge}. We do this by proving the corresponding Laplace upper and lower bounds, using arguments completely analogous to those in \cite{BD19}; we include the proofs for completeness.%

\begin{proposition}[Laplace upper bound]
\label{prop:upper}
Suppose Assumptions \ref{ass:b-sigma-local-properties}-\ref{ass:log-gradient-convergence} hold. Then, for any bounded, continuous $F$,
\begin{align}
\label{eq:Laplace-upper-bound}
\liminf _{\eta \to 0} -\eta \log \mathbb{E} \left[ \exp\left\{ -\frac{1}{\eta} F(X^{\eta, xy}) \right\} \right] \geq  \inf _{\varphi \in \calC _1} \left[ F(\varphi) + \inf _{\calU _\varphi} \frac{1}{2} \int _0 ^1 | \nu _s | ^2 ds  \right].
\end{align}
\end{proposition}

\begin{proposition}[Laplace lower bound]
\label{prop:lower}
Suppose Assumptions \ref{ass:b-sigma-local-properties}-\ref{ass:log-gradient-convergence} hold. Then, for any bounded, continuous $F$,
\begin{align}
\label{eq:Laplace-lower-bound}
\limsup _{\eta \to 0} -\eta \log \mathbb{E} \left[ \exp\left\{ -\frac{1}{\eta} F(X^{\eta, xy}) \right\} \right] \leq  \inf _{\varphi \in \calC _1} \left[ F(\varphi) + \inf _{\calU _\varphi} \frac{1}{2} \int _0 ^1 | \nu _s | ^2 ds  \right].
\end{align}
\end{proposition}
The combination of the two gives the Laplace principle of Theorem \ref{thm:LDP_bridge}: for bounded, continuous $F$,
\begin{equation}\label{eq:Laplace-principle}
    \lim_{\eta \downarrow 0} -\eta \log \bE\left[\exp{-\frac{1}{\eta}F(X^{\eta, xy})}\right] = \inf_{\varphi \in C([0,1]:\mathbb{R}^d)} \left[F(\varphi) + \inf_{\nu \in U^{xy}_\varphi} \frac{1}{2}\int_0^1|\nu_s|^2ds\right].
\end{equation}
This establishes the claimed LDP with rate function given by
\begin{equation*}
    I_D^{xy}(\varphi) = I^{xy}(\varphi) \coloneq \inf_{\nu \in U^{xy}_\varphi} \frac{1}{2}\int_0^1|\nu_s|^2ds.
\end{equation*}
It is also straightforward to see that $I^{xy}$ is a good rate function, that is it has compact level sets, using arguments similar to those in sections \ref{sec:convergence-to-controlled-ode} and the proof of Proposition \ref{prop:lower}; see \cite[Section 3.2.4]{BD19}. 

We end this section with proofs of the two Laplace bounds. Recall from Proposition \ref{prop:BBrepM} that we have the $M$/$\delta$ version of the representational formula \eqref{eq:representation}: for every $\delta > 0$ there is a number $M < \infty$ and controls $\{\nu^\eta\}_{\eta \in (0,1)}$, that all spend less than $M$ energy, such that for every $\eta \in (0,1)$,
\begin{equation}\label{eq:representation-M-delta}
    -\eta \log \bE \exp \left\{-\frac{1}{\eta} F(X^{\eta, xy})\right\} + \delta \geq \bE \left[F(\bar{X}^{\eta, xy}) + \frac{1}{2}\int_0^1 | \nu^\eta_s |^2 ds \right].
\end{equation}
Here, $\bar{X}^{\eta, xy}$ follows the controlled SDE
\begin{equation*}
    d\bar{X}^{\eta, xy}_t = (b + g_\eta^y)(t, \bar{X}^{\eta, xy}_t) dt + \sqrt{\eta} \sigma(t, \bar{X}^{\eta, xy}_t) dW_t + \sigma(t, \bar{X}^{\eta, xy}_t) \nu^\eta_t dt, \ \ \bar{X}^{\eta, xy}_0 = x.
\end{equation*}
We have shown that if $(\bar{X}^{\eta, xy}, \nu^\eta)$ converges weakly, as $\eta \downarrow 0$, to $(\bar{X}^{xy}, \nu)$, then $(\bar{X}^{xy}, \nu)$ follows 
\begin{equation}\label{eq:noise-free-controlled-BB-00}
    d\bar{X}^{xy}_t = (b + g^y)(t, \bar{X}^{xy}_t) dt + \sigma(t, \bar{X}^{xy}_t) \nu_t dt, \ \ \bar{X}^{ xy}_0 = x.
\end{equation}
We now use these facts to show the Laplace principle upper and lower bounds.

\begin{proof}[Proof of Laplace upper bound]

Just like in \cite{BD19}, we have merely shown that $(\bar{X}^{\eta, xy}, \nu^\eta)_{\eta}$ is tight, not that it converges.
The limit behavior above still applies to any convergent subsequence.
For any subsequence of $(\bar{X}^{\eta, xy}, \nu^\eta)_{\eta}$ such that the $\liminf$ on the right-hand side in \eqref{eq:Laplace-upper-bound} is a limit, take a further convergent (to $(\bar{X}^{xy}, \nu)$) subsequence.
Then we have, denoting all subsequences by $\eta$,

\begin{equation*}%
\begin{split}
    &\liminf_{\eta \downarrow 0} -\eta \log \bE \exp \left\{-\frac{1}{\eta} F(X^{\eta, xy})\right\} + 2\delta \\
    &\qquad \geq \liminf_{\eta \downarrow 0} \bE \left[F(\bar{X}^{\eta, xy}) + \frac{1}{2}\int_0^1 | \nu^\eta_s |^2 ds \right] \\
    \quad &\qquad \geq \bE \left[F(\bar{X}^{xy}) + \frac{1}{2}\int_0^1 | \nu_s |^2 ds \right] \\
    &\qquad \geq \bE \left[F(\bar{X}^{xy}) + \inf_{\nu \in U_{\bar{X}^{xy}}} \frac{1}{2}\int_0^1 | \nu_s |^2 ds \right] \\
    &\qquad \geq \inf_{\varphi \in C([0,1]: \mathbb{R}^d)} \Big[F(\varphi) + \inf_{\nu \in U_{\varphi}} \frac{1}{2}\int_0^1 | \nu_s |^2 ds\Big].
\end{split}
\end{equation*}

\noindent Since $\delta$ is arbitrarily small, the Laplace upper bound \eqref{eq:Laplace-upper-bound} follows.
\end{proof}

\begin{proof}[Proof of Laplace lower bound]
Choose $\varphi^*$ such that $F(\varphi^*) + I(\varphi^*) \leq \inf_{\varphi} F(\varphi) + I(\varphi) + \delta$. 
Also, choose $u \in U_{\varphi^*}$ such that $\frac{1}{2}\int_0^1 |u_s|^2 ds \leq I(\varphi^*) + \delta$.
Then, taking $\tilde{X}^{\eta, xy}$ to be the $\eta$-noise process controlled by $u$,
consider a convergent subsequence of a $\limsup$-attaining subsequence of $\tilde{X}^{\eta, xy}$.
By Section \ref{sec:convergence-to-controlled-ode}, $(\tilde{X}^{\eta, xy}, u)$ converges to $(\varphi^*, u)$.
From this we obtain the upper bound
\begin{equation*}%
\begin{split}
    & \limsup_{\eta \downarrow 0} -\eta \log \bE \exp \left\{-\frac{1}{\eta} F(X^{\eta, xy})\right\}  \\
    & \qquad = \limsup_{\eta \downarrow 0} \inf_{\nu^\eta \in \mathcal{A}} \bE \left[F(\bar{X}^{\eta, xy}) + \frac{1}{2}\int_0^1 | \nu^\eta_s |^2 ds \right] \\
    &\qquad \leq \limsup_{\eta \downarrow 0} \bE \left[F(\tilde{X}^{\eta, xy}) + \frac{1}{2}\int_0^1 | u_s |^2 ds \right] \\
    &\qquad \leq \bE \left[F(\varphi^*) + \frac{1}{2}\int_0^1 | u_s |^2 ds \right] \\
    &\qquad = F(\varphi^*) + \frac{1}{2}\int_0^1 | u_s |^2 ds \\
    &\qquad \leq F(\varphi^*) + \inf_{\nu \in U_{\varphi^*}} \frac{1}{2}\int_0^1 | \nu_s |^2 ds + \delta \\
    &\qquad \leq \inf_{\varphi \in C([0,1]: \mathbb{R}^d)} \left[F(\varphi) + \inf_{\nu \in U_{\varphi}} \frac{1}{2}\int_0^1 | \nu_s |^2 ds\right] + 2\delta. \\
\end{split}
\end{equation*}
Since $\delta >0$ is arbitrarily small, the Laplace lower bound \eqref{eq:Laplace-lower-bound} follows.
\end{proof}

\subsection{Uniformity of Laplace principle on compacts}
\label{sec:uniform}

This section is aimed at proving the uniform Laplace principle of Theorem \ref{thm:uniform-laplace-principle}. The Laplace limit, established in Section \ref{sec:development}, for each $(x,y)$ is the function $\tilde{F}: D^2 \to \bR$, defined by
\begin{equation}
    \tilde{F}(x,y) \coloneq \inf_{\varphi \in \calC_1} F(\varphi) + I^{xy}(\varphi).
\end{equation}
In order to prove the uniform Laplace principle, we first need the continuity of $\tilde F$. %
\begin{proposition}\label{prop:F-tilde-continuous}
    Under Assumptions \ref{ass:b-sigma-local-properties}-\ref{ass:log-gradient-convergence}$, \tilde{F}$ is continuous.
\end{proposition}
The proof of Proposition \ref{prop:F-tilde-continuous} is given at the end of this section.
With this continuity established, we can now give the proof of Theorem \ref{thm:uniform-laplace-principle}.  %

\begin{proof}[Proof of Theorem \ref{thm:uniform-laplace-principle}]
    Consider a compact set $K \Subset D$.
    The main step is to show that if pairs $(x, y)$ and $(x', y')$ are selected from this set, then, uniformly over $K$, if they are close enough, then asymptotically so are the resulting quantities $\tilde F_\eta(x, y) $ and $\tilde F_\eta(x', y')$, where
    \[ 
    \tilde F_\eta(x, y) \coloneq -\eta \log \bE\left[\exp - \frac{1}{\eta} F(X^{\eta, xy}) \right].\]  
    This, when combined with Proposition \ref{prop:F-tilde-continuous}, will yield the uniform Laplace principle.

    Assume w.l.o.g. that $\tilde F_\eta(x, y) \leq \tilde F_\eta(x', y')$.
    Using the representation formula, we have for each $\eta$ that there is some control $\nu$ such that
    \begin{equation}
    \begin{split}
         -\eta \log \bE\left[\exp - \frac{1}{\eta} F(X^{\eta, xy}) \right] &= \inf_{\tilde{\nu}} \bE\left[F({\bar X}^{\eta, xy}) + \frac{1}{2} \int_0^1 |\tilde \nu_s|^2 ds \right] \\
         &\geq \bE\left[F({\bar X}^{\eta, xy}) + \frac{1}{2} \int_0^1 |\nu_s|^2 ds \right] - \varepsilon,
    \end{split}
    \end{equation}
    for some arbitrarily small $\varepsilon$, while,
    \begin{equation*}
        \frac{1}{2} \int_0^1 |\nu_s|^2 ds \leq 2 \norm{F}_\infty + 1.
    \end{equation*}

    For a large enough compact set $K_{+} \supseteq K$, we will get that for $\eta \leq \eta_0$, \[\bP(\tau^{\eta, \tilde x \tilde y}_{K_+} < \infty) \leq e^{-\frac{2((2 \norm{F}_\infty + 1)+1)}{\eta}}\] for all $x,y \in K$ by Assumption \ref{ass:exp-boundedness}. %
    Now, take the same $\nu$ and apply it to the representation for $(x', y')$ to obtain
    \begin{equation}\label{eq:ULDP-difference-to-representation-epsilon}
    \begin{split}
        0 &\leq \tilde F_\eta(x', y') - \tilde F_\eta(x, y) \\
        &=-\eta \log \bE\left[\exp - \frac{1}{\eta} F(X^{\eta, x'y'}) \right] - -\eta \log \bE\left[\exp - \frac{1}{\eta} F(X^{\eta, xy}) \right] \\ 
        &\leq \bE\left[F({\bar X}^{\eta, x'y'}) + \frac{1}{2} \int_0^1 |\nu_s|^2 ds \right] - \bE\left[F({\bar X}^{\eta, xy}) + \frac{1}{2} \int_0^1 |\nu_s|^2 ds \right] + \varepsilon \\
        &= \bE\left[F({\bar X}^{\eta, x'y'}) - F({\bar X}^{\eta, xy})\right] + \varepsilon. \\
    \end{split}
    \end{equation}

    The arguments of Section \ref{sec:tightness-of-controlled-via-Lipschitz} that showed tightness of $\{\bar X^{\eta, xy}\}_{\eta > 0}$ are also valid for the tightness of $\{\bar X^{\eta, xy}\}_{\eta > 0, (x,y) \in K^2}$.
    Then, we may take a compact set $A \in \calC_1$ such that 
    \begin{equation}\label{eq:A-C1-complement-probability-bound}
        \bP((X^{\eta, xy}, X^{\eta, x'y'} \in A)^c) \leq \varepsilon / (2\norm{F}_\infty).
    \end{equation}
    Since $F$ is continuous, there is a $\lambda > 0$ such that, for all $\varphi_1, \varphi_2 \in A$, $\norm{\varphi_1 - \varphi_2}_\infty < \lambda$ implies that $|F(\varphi_1) - F(\varphi_2)| < \varepsilon$.
    Choose $\delta$ in Lemma \ref{lem:controlled-bridges-near-y} such that 
    \begin{equation*}
        \lim_{\eta \downarrow 0} \bP\left(\sup_{t \in [1-\delta, 1]} |\bar{X}^{\eta, xy}_t - y| \geq \frac{\lambda}{3}\right) = 0.
    \end{equation*}
    Now consider the difference between the two processes up to $t \leq 1 - \delta$ on the event $X^{\eta, xy}, X^{\eta, x'y'} \in A$.
    \begin{equation}\label{eq:long}
    \begin{split}
        &\sup_{s \leq t} |{\bar X}^{\eta, x'y'}_s - {\bar X}^{\eta, xy}_s| \\
        &\qquad = \sup_{s \leq t} \Bigl|x' - x + \int_0^s (b(u, \bar X^{\eta, x'y'}_u) - b(u, \bar X^{\eta, xy}_u) + g^{y'}(u, \bar X^{\eta, x'y'}_u) - g^{y}(u, \bar X^{\eta, xy}_u))du \\
        &\qquad\qquad + \int_0^s \sqrt{\eta}(\sigma(u, \bar X^{\eta, x'y'}_u) - \sigma(u, \bar X^{\eta, xy}_u))(dW_u + \nu_u du)\Bigr| \\
        &\qquad \leq |x' - x| + \sup_{s \leq t} \int_0^s \Bigl(|b(u, \bar X^{\eta, x'y'}_u) - b(u, \bar X^{\eta, xy}_u)| + |g^{y'}_\eta(u, \bar X^{\eta, x'y'}_u) - g^{y'}_\eta(u, \bar X^{\eta, xy}_u)| \\
        &\qquad\qquad\qquad + |g^{y'}_\eta(u, \bar X^{\eta, xy}_u) - g^{y}_\eta(u, \bar X^{\eta, xy}_u)| + |(\sigma(u, \bar X^{\eta, x'y'}_u) - \sigma(u, \bar X^{\eta, xy}_u))\nu_u|\Bigr)du \\
        &\qquad\qquad + \Bigl|\int_0^s \sqrt{\eta}(\sigma(u, \bar X^{\eta, x'y'}_u) - \sigma(u, \bar X^{\eta, xy}_u))dW_u\Bigr| \\
        &\qquad \leq |x' - x| + \sup_{s \leq t} \int_0^s \Bigl(2L|\bar X^{\eta, x'y'}_u - \bar X^{\eta, xy}_u| %
        + L|y' - y| + 2L|\bar X^{\eta, x'y'}_u - \bar X^{\eta, xy}_u||\nu_u|\Bigr)du \\
        &\qquad\qquad + \Bigl|\int_0^s \sqrt{\eta}(\sigma(u, \bar X^{\eta, x'y'}_u) - \sigma(u, \bar X^{\eta, xy}_u))dW_u\Bigr| \\
        &\qquad \leq (1+L)\Delta + 2L \int_0^t \Bigl(\sup_{s \leq u} |\bar X^{\eta, x'y'}_s - \bar X^{\eta, xy}_s| + \sup_{s \leq u} |\bar X^{\eta, x'y'}_s - \bar X^{\eta, xy}_s||\nu_u|\Bigr)du \\
        &\qquad\qquad + \Bigl|\int_0^s \sqrt{\eta}(\sigma(u, \bar X^{\eta, x'y'}_u) - \sigma(u, \bar X^{\eta, xy}_u))dW_u\Bigr| \\
    \end{split} 
    \end{equation}
    where we used the local Lipschitz properties of $b$ and $g^{\tilde y}$, letting $L$ be a common Lipschitz constant (also for $\sigma$) with respect to $\delta$ and $A$ chosen above, in the penultimate step, and defining $\Delta \coloneq |x' - x| + |y' - y|$.
    
    We can bound the control term above, using the Cauchy-Schwarz inequality, as 
    \begin{equation*}
    \begin{split}
        &\int_0^t \sup_{s \leq u} |X^{\eta, x'y'}_s - X^{\eta, xy}_s||\nu_u|du \\
        &\qquad \leq \Bigl(\int_0^t \sup_{s \leq u} |\bar X^{\eta, x'y'}_s - \bar X^{\eta, xy}_s|^2 du \int_0^t|\nu_u|^{2}du \Bigr)^{\frac{1}{2}} \\
        &\qquad \leq (2(2\norm{F}_\infty+1))^\frac{1}{2} \Bigl(\int_0^t \sup_{s \leq u} |\bar X^{\eta, x'y'}_s - \bar X^{\eta, xy}_s|^2 du \Bigr)^{\frac{1}{2}}. \\
    \end{split}
    \end{equation*}
    Taking the square of \eqref{eq:long}, and using Jensen's inequality to get it under the integral sign, we have for an appropriate constant $C$ that
    \begin{equation}\label{eq:long-cont}
    \begin{split}
        &\sup_{s \leq t} |{\bar X}^{\eta, x'y'}_s - {\bar X}^{\eta, xy}_s|^2 \\
        &\qquad \leq C\Delta + C \int_0^t \sup_{s \leq u} |\bar X^{\eta, x'y'}_s - \bar X^{\eta, xy}_s|^2 du \\
        &\qquad\qquad + C \Bigl|\int_0^s \sqrt{\eta}(\sigma(u, \bar X^{\eta, x'y'}_u) - \sigma(u, \bar X^{\eta, xy}_u))dW_u\Bigr|^2 \\
    \end{split} 
    \end{equation}
    The last right-hand side stochastic integral term may be bounded with high probability ``on $(X^{\eta, xy}, X^{\eta, x'y'} \in A)$'', using \eqref{eq:classical-estimate-d}, since both $\sigma$-terms are bounded by $L$ in Frobenius norm, so in each dimension $i$, we get that 
    \begin{equation*}
        \langle \int_0^{\cdot} \sqrt{\eta}(\sigma(u, \bar X^{\eta, x'y'}_u) - \sigma(u, \bar X^{\eta, xy}_u))_i dW_u \rangle_t \leq \eta \int_0^{t} L |\bar X^{\eta, x'y'}_u - \bar X^{\eta, xy}_u| du.
    \end{equation*}
    The last integral is bounded by some constant, we may assume by $C$.
    Then, for small enough $\eta$, the last term in \eqref{eq:long-cont} is bounded, e.g. by $C\Delta$ with probability going to $1$, and on these events,
    \begin{equation*}
    \begin{split}
        &\sup_{s \leq t} |{\bar X}^{\eta, x'y'}_s - {\bar X}^{\eta, xy}_s|^2 \leq 2C\Delta + C \int_0^t \sup_{s \leq u} |\bar X^{\eta, x'y'}_s - \bar X^{\eta, xy}_s|^2 du.
    \end{split} 
    \end{equation*}
    Now setting $\Delta$ small enough, we get by Gr\"onwall's inequality that, with probability going to $1$,
    \begin{equation*}
        \sup_{t \leq 1-\delta} |{\bar X}^{\eta, x'y'}_t - {\bar X}^{\eta, xy}_t|^2 \leq \lambda.
    \end{equation*}
    With $\bP(\sup_{t \in [1-\delta, 1]} |\bar{X}^{\eta, xy}_t - \tilde y| \geq \lambda / 3) \to 0$ for both $\tilde y = y$ and $\tilde y = y'$, and assuming w.l.o.g. that $\Delta \leq \lambda / 3$, we get the same conclusion for $t \geq 1-\delta$ using the triangle inequality, and thus
    this extends to $t \leq 1$.
    In other words,
    \begin{equation*}
        \sup_{t \leq 1} |{\bar X}^{\eta, x'y'}_t - {\bar X}^{\eta, xy}_t|^2 \leq \lambda,
    \end{equation*}
    with probability going to $1$ as $\eta \downarrow 0$, uniformly over $K$.
    Thus,
    \begin{equation*}
    \begin{split}
        &\bE\left[|F({\bar X}^{\eta, x'y'}) - F({\bar X}^{\eta, xy})| \mathbbm{1}_{\bar \tau^{\eta}_{K_+} = \infty} \right] \\
        &\qquad \leq \varepsilon + \bE\left[|F({\bar X}^{\eta, x'y'}) - F({\bar X}^{\eta, xy})| \mathbbm{1}_{\bar \tau^{\eta}_{K_+} = \infty}\mathbbm{1}_{\sup_{t \leq 1} |{\bar X}^{\eta, x'y'}_t - {\bar X}^{\eta, xy}_t|^2 > \lambda} \right] \\
        &\qquad \leq \varepsilon + \bP\left({\bar \tau^{\eta}_{K_+} = \infty}, \, \sup_{t \leq 1} |{\bar X}^{\eta, x'y'}_s - {\bar X}^{\eta, xy}_s|^2 > \lambda \right) \\
        &\qquad \leq 2\varepsilon,
    \end{split}
    \end{equation*}
    for all small enough $\eta$, uniformly over all $(x,y), (x', y') \in K$ with distance $\Delta$ small enough.
    Combined with \eqref{eq:ULDP-difference-to-representation-epsilon} and \eqref{eq:A-C1-complement-probability-bound}, this gives 
    \begin{equation*}
        |\tilde F_\eta(x', y') - \tilde F_\eta(x, y)| \leq 4 \varepsilon
    \end{equation*}
    for all small enough $\eta$.
    Since $\varepsilon$ is arbitrarily small, %
    we can say
    that for any $\varepsilon > 0$ there are small enough $\Delta > 0$ and $\eta_0 > 0$ such that, for any two pairs $(x,y), (x',y') \in K \times K$, $|\tilde F_\eta(x', y') - \tilde F_\eta(x, y)| \leq \varepsilon$ for all $\eta \leq \eta_0$, whenever $|x' - x| + |y' - y| \leq \Delta$.

    This conclusion leads to the uniform Laplace principle:
    For $\varepsilon > 0$, choose a radius $\Delta$ chosen small enough so that the above holds, which also gives us that $|\tilde F(z') - \tilde F(z)| \leq \varepsilon$ whenever $|z' - z| \leq \Delta$.
    Cover $K \times K$ by finitely many balls of radius $\Delta$, with centers $\{z_i\}_{i=1}^{N_\Delta}$.
    Then, take an $\eta_1 \leq \eta_0$ such that $|\tilde F_\eta(z_i) - \tilde F(z_i)| \leq \varepsilon$ for all centers $z_i$.
    Then, for any $z \in K \times K$, there is a ball $B_{\bR^d \times \bR^d}(z_i, \Delta)$ to which it belongs, meaning that for $\eta \leq \eta_1$,
    \begin{equation*}
    \begin{split}
        |\tilde F_\eta(z) - \tilde F(z)| &\leq |\tilde F_\eta(z) - \tilde F_\eta(z_i)| + |\tilde F_\eta(z_i) - \tilde F(z_i)| + |\tilde F(z_i) - \tilde F(z)| \\
        &\leq 3\varepsilon.
    \end{split}
    \end{equation*}
    Since $\varepsilon$ can be made arbitrarily small, and $\eta_1$ is chosen independently of $z \in K \times K$, we have the uniform Laplace principle:
    \begin{equation*}
        \lim_{\eta \downarrow 0} \sup_{z \in K \times K} |\tilde F_\eta(z) - \tilde F(z)| = 0.
    \end{equation*}
\end{proof}

With the proof of Theorem \ref{thm:uniform-laplace-principle} complete, we now prove the necessary Proposition \ref{prop:F-tilde-continuous}, that establishes the continuity of $\tilde F$. This in turn relies on the following Lemma. %
\begin{lemma}\label{lem:controlled-bridges-near-y}
    If $K \Subset D$, then for any $\xi > 0$ there exists $\delta > 0$ such that for all $x, y \in K$, $\bP(\sup_{t \in [1-\delta, 1]} |\bar{X}^{\eta, xy}_t - y| \geq \xi) \to 0$ as $\eta \downarrow 0$, and \[\bP\left(\sup_{t \in [1-\delta, 1]} |\bar{X}^{xy}_t - y| \geq \xi\right) = 0.\]
\end{lemma}

\begin{proof}[Proof of Lemma \ref{lem:controlled-bridges-near-y}]
    Using the same technique as before, we see that
    \begin{equation}
    \begin{split}
        \bP(\sup_{t \in [1-\delta, 1]} |\bar{X}^{\eta, xy}_t - y| \geq \xi) \leq e^{\frac{M}{\eta}} \bP(\sup_{t \in [1-\delta, 1]} |X^{\eta, xy}_t - y| \geq \xi),
    \end{split}
    \end{equation}
    so what needs to be shown is that $\bP(\sup_{t \in [1-\delta, 1]} |X^{\eta, xy}_t - y| \geq \xi) = o(e^{-\frac{M}{\eta}})$.
    For a large enough compact set $K_{+} \supseteq K$, we will get that for $\eta \leq \eta_0$, $\bP(\tau^{\eta, xy}_{K_+} < \infty) \leq e^{-\frac{M+1}{\eta}}$ for all $x,y \in K$ by Assumption \ref{ass:exp-boundedness}, and thus we may condition on $\tau^{\eta, xy}_{K_+} = \infty$.
    Studying a reversed bridge $\hat{X}^{\eta, xy}_t$ on this event, this means it is equal to their stopped versions $\hat{X}^{\eta, xy}_{t \land \hat \tau^{\eta, xy}_{K_+}}$, with the stopping time $\tau^{\eta, xy}_{K_+} \coloneq \inf \{t \geq 0: \hat{X}^{\eta, xy}_t \notin K_+\}$.
    In the reversed dynamics, \ref{eq:reverse-xy}, we get then that $b(1 - t, \hat{X}^{\eta, xy}_{t \land \hat \tau^{\eta, xy}_{K_+}})$ and $\overset{\leftarrow}{g}{}^x_\eta(t, \hat{X}^{\eta, xy}_{t \land \hat \tau^{\eta, xy}_{K_+}})$, and $\sigma(1-t, \hat{X}^{\eta, xy}_{t \land \hat \tau^{\eta, xy}_{K_+}})$ are bounded by some constant $C_\mathrm{max}$ for e.g. $t \in [0, \frac{1}{2}]$.
    Thus, for $\delta \leq \frac{1}{2}$,
    \begin{equation*}
    \begin{split}
        &\bP\Bigl(\sup_{t \in [1-\delta, 1]} |X^{\eta, xy}_t - y| \geq \xi, \tau^{\eta, xy}_{K_+} = \infty\Bigr) \\
        &\qquad = \bP\Bigl(\sup_{t \in [0, \delta]} |\hat X^{\eta, xy}_t - y| \geq \xi, \hat\tau^{\eta, xy}_{K_+} = \infty) \\
        &\qquad \leq \bP\Bigl(\sup_{t \in [0, \delta]} \Bigr|\int_0^t (-b(1 - s, \hat{X}^{\eta, xy}_{s \land \hat \tau^{\eta, xy}_{K_+}}) + \overset{\leftarrow}{g}{}^x_\eta(s, \hat{X}^{\eta, xy}_{s \land \hat \tau^{\eta, xy}_{K_+}})) ds \\
        &\qquad\qquad + \int_0^t \sqrt{\eta} \sigma(1 - s, \hat{X}^{\eta, xy}_{s \land \hat \tau^{\eta, xy}_{K_+}}) d\hat{W}_s \Bigr| \geq \xi\Bigr) \\
        &\qquad\leq \bP\Bigl(\sup_{t \in [0, \delta]} \Bigl|\int_0^t \sqrt{\eta} \sigma(1 - s, \hat{X}^{\eta, xy}_{s \land \hat \tau^{\eta, xy}_{K_+}}) d\hat{W}_s \Bigr| \geq \xi - 2C_\mathrm{max}\delta\Bigr).
    \end{split}
    \end{equation*}
    The quadratic variation of the last stochastic integral is bounded by $\eta C_\mathrm{max} \delta$ in each dimension.
    Using the inequality \eqref{eq:classical-estimate-d}, we get
    \begin{equation*}
    \begin{split}
        &\bP\Bigl(\sup_{t \in [1-\delta, 1]} |X^{\eta, xy}_t - y| \geq \xi, \tau^{\eta, xy}_{K_+} = \infty\Bigr) \\
        &\qquad\leq \bP\Bigl(\sup_{t \in [0, \delta]} \Bigl|\int_0^t \sqrt{\eta} \sigma(1 - s, \hat{X}^{\eta, xy}_{s \land \hat \tau^{\eta, xy}_{K_+}}) d\hat{W}_s \Bigr| \geq \xi - 2C_\mathrm{max}\delta\Bigr) \\
        &\qquad\leq 2d\exp-\frac{((\xi - 2C_\mathrm{max}\delta)/\sqrt{d})^2}{2(\eta C_\mathrm{max}\delta)^2} \\
        &\qquad\leq 2d\exp-\frac{(\frac{1}{2}\xi/\sqrt{d})^2}{2(\eta C_\mathrm{max}\delta)} = 2d\exp-\frac{\xi^2}{8d\eta C_\mathrm{max}\delta},
    \end{split}
    \end{equation*}
    where the last inequality holds for small enough $\delta$.
    The last expression is clearly $o(e^{-\frac{M}{\eta}})$ for some small enough $\delta > 0$.

    Combined, we get that $\bP(\sup_{t \in [1-\delta, 1]} |\bar{X}^{\eta, xy}_t - y| \geq \xi) \to 0$ as $\eta \downarrow 0$.
    Specifically, this holds for sequences converging to $(\bar X^{xy}, \nu)$.
    By consequence, we have that $\bP(\sup_{t \in [1-\delta, 1]} |\bar{X}^{xy}_t - y| \geq \xi) = 0$.
\end{proof}

\begin{proof}[Proof of Proposition \ref{prop:F-tilde-continuous}]
    Assume wlog that $\tilde{F}(x,y) \leq \tilde{F}(x', y')$.
    Let $\Delta = |x - x'| + |y - y'|$.
    Let $\varphi \in \calC_1$ and $\nu \in U^{xy}_\varphi$ be such that 
    \begin{equation*}
        \tilde{F}(x,y) + \delta_F \geq F(\varphi) + \frac{1}{2}\int_0^1 |\nu_s|^2 ds,
    \end{equation*}
    where we may assume that the second term is smaller than $2\norm{F}_\infty + 1$.
    Now, let $\varphi'$ be defined on $[0,1)$ by the ODE:
    \begin{equation*}
        \varphi'_t = x' + \int_0^{t} (b(s, \varphi'_s) + g^{y'}(s, \varphi'_s) + \sigma(s, \varphi'_s) \nu_s) ds.
    \end{equation*}
    Then $\nu \in U^{x'y'}_{\varphi'}$.
    Now wish to show that for a small $\Delta$, $\varphi'$ is close to $\varphi$ (in sup-norm).
    Chose $\delta > 0$ from Lemma \ref{lem:controlled-bridges-near-y} such that $\bP(\sup_{t \in [1-\delta, 1]} |\bar{X}^{\tilde x \tilde y}_t - \tilde y| \geq \xi) = 1$ for all $(\tilde x, \tilde y) \in B_{\bR^d \times \bR^d}((x,y), \Delta)$.
    We have that for $s \in [0,1)$:
    \begin{equation*}
    \begin{split}
        |\varphi_s - \varphi'_s| &= \Bigl|x - x' + \int_0^{s} (b(u, \varphi_u) + g^{y}(u, \varphi_u) + \sigma(u, \varphi_u) \nu_u) du \\
        &\qquad - \int_0^{s} (b(u, \varphi'_u) + g^{y'}(u, \varphi'_u) + \sigma(u, \varphi'_u) \nu_u) du \Bigr| \\
        &\leq \Delta + \int_0^{s} |b(u, \varphi_u) - b(u, \varphi'_u)|du + \int_0^s |g^{y}(u, \varphi_u) - g^{y'}(u, \varphi'_u)|du \\
        & \qquad + \int_0^s |(\sigma(u, \varphi_u) - \sigma(u, \varphi'_u)) \nu_u)| du. \\
        &\leq \Delta + \int_0^{s} |b(u, \varphi_u) - b(u, \varphi'_u)|du + \int_0^s |g^{y}(u, \varphi_u) - g^{y'}(u, \varphi_u)|du \\
        & \qquad + \int_0^s |g^{y'}(u, \varphi_u) - g^{y'}(u, \varphi'_u)|du + \int_0^s |(\sigma(u, \varphi_u) - \sigma(u, \varphi'_u)) \nu_u)| du. \\
    \end{split}
    \end{equation*}
    For a sufficiently large $K_\alpha$, we have that $\varphi, \varphi'$ are bounded to $K_\alpha$ with probability $1$.
    Then, $b$, $g^{\cdot}$ (in both its spatial arguments), and $\sigma$ are Lipschitz on $[0, 1-\delta] \times K_\alpha$ with some common constant $L$.
    Inserting it above gives, for $s \in [0, 1-\delta]$,
    \begin{equation*}
    \begin{split}
        |\varphi_s - \varphi'_s| &\leq \Delta + \int_0^{s} 3L|\varphi_u - \varphi'_u|du + \int_0^s L |\varphi_u - \varphi'_u| |\nu_u| du. \\
        &\leq \Delta + \int_0^{s} 3L \sup_{u \leq s}|\varphi_u - \varphi'_u|du + \int_0^s L \sup_{u \leq s}|\varphi_u - \varphi'_u| |\nu_u| du \\
        &\leq \Delta + L \sup_{u \leq s}|\varphi_u - \varphi'_u| (3 + \int_0^s |\nu_u| du).
    \end{split}
    \end{equation*}
    Using the inequality $(a + b)^2 \leq 2a^2 + 2b^2$ twice, we have 
    \begin{equation*}
    \begin{split}
        \sup_{u \leq s} |\varphi_u - \varphi'_u|^2 %
        &\leq 2 \Delta + 2 L^2 \sup_{u \leq s}|\varphi_u - \varphi'_u|^2 (3 + \int_0^s |\nu_u| du)^2 \\
        &\leq 2 \Delta + 2 L^2 \sup_{u \leq s}|\varphi_u - \varphi'_u|^2 (18 + 2(\int_0^s |\nu_u| du)^2) \\
        &\leq 2 \Delta + 2 L^2 \sup_{u \leq s}|\varphi_u - \varphi'_u|^2 (18 + 2\int_0^s |\nu_u|^2 du) \\
        &\leq 2 \Delta + 2 L^2 \sup_{u \leq s}|\varphi_u - \varphi'_u|^2 (18 + 4(2\norm{F}_\infty + 1)). \\
    \end{split}
    \end{equation*}
    Gr\"onwall's inequality, gives that, on $[0, 1-\delta]$, with $\beta \coloneq 2 L^2 (18 + 4(2\norm{F}_\infty + 1))$,
    \begin{equation*}
        \sup_{u \leq s} |\varphi_s - \varphi'_s|^2 \leq 2 \Delta e^{\beta s},
    \end{equation*}
    and, in particular, $\sup_{u \leq 1 - \delta} |\varphi_u - \varphi'_u|^2 \leq 2 \Delta e^{\beta (1-\delta)}$.
    Thus, selecting $\Delta$ small enough yields that $\sup_{u \in [0,1]} |\varphi_u - \varphi'_u|^2 \leq \varepsilon$.
\end{proof}

\section{Laplace principle of the dynamic Schr\"odinger problem}
\label{sec:fullLDP}

With the results of Section \ref{sec:development}, we are ready to ready to prove Theorem \ref{thm:LDP_dSB}, the Laplace principle for $\{\pi_\eta\}_{\eta > 0}$ solving the dynamic Schrödinger problem
\begin{equation*}
    \inf_{\pi \in \Pi^{\calC_1}(\mu, \nu)} \calH(\pi \mid\mid R_\eta).
\end{equation*}
The strategy is to combine the LDP for the static problem with the uniform Laplace principle of the dynamic bridges.

Similar to Section \ref{sec:LaplaceBridge}, we prove the full Laplace principle by proving the corresponding upper and lower bounds. This is done in Sections \ref{sec:laplace-upper-bound} and \ref{sec:laplace-lower-bound}, respectively. Moreover, we identify that the (good) rate function is given by
\begin{align*}
    I_D(\varphi) &= I_{S}(\varphi_0,\varphi_1) + I_B^{\varphi_0 \varphi_1}(\varphi) \\
    &= (c-(-\psi \oplus \phi))(\varphi_0, \varphi_1) + \inf_{\nu \in U^{xy}_\varphi} \frac{1}{2} \int_0^1 |\nu(t)|^2 dt.
\end{align*}

Before we embark on the proofs, some notation and preliminary steps. With $\pi _\eta$ the solution of the dynamic SBP for $R _\eta$ (assumed to exist), we denote by $\pi^{\eta, xy}_{\cdot | 01}$ the stochastic kernel given by a disintegration of $\pi_{\eta}$, when conditioned on $(x,y)$, and similarly for an $m \in \calP(\calC_1)$.
Let $F$ be a bounded continuous functional $F: \calC_1 \to \bR$ and let $X^\eta \sim \pi^\varepsilon$.
Then, using a standard representation formula for $\log \bE \exp F(X)$ for continuous functionals $F$ \cite{BD19} and disintegration, we obtain the following useful expression for $-\eta \log \bE e^{-F(X^\eta) /\eta}$,
\begin{equation*}
\begin{split}
    -\eta \log \bE e^{ -\frac{1}{\eta} F(X^{\eta})} &= \inf_{m \in \mathcal{P}(\calC_1)} \int F dm + \eta \calH(m \mid\mid \pi_\eta) \\
    &= \inf_m \int m_{01}(dx, dy) \int F dm_{\cdot | 01}^{xy} \\
    & \qquad + \eta \left(\calH(m_{01} \mid\mid \pi_{\eta, 01}) + \int m_{01}(dx, dy) \calH(m_{\cdot | 01}^{xy} \mid\mid \pi^{\eta, xy}_{\cdot | 01})\right) \\
    &= \inf_m \int m_{01}(dx, dy) \left(\int F dm_{\cdot | 01}^{xy} + \eta \calH(m_{\cdot | 01}^{xy} \mid\mid \pi^{\eta, xy}_{\cdot | 01})\right) \\
    & \qquad + \eta \calH(m_{01} \mid\mid \pi_{\eta, 01}) \\
    &= \inf_{m_{01} \in \calP(\bR^{2d})} \inf_{m_{\cdot | 01}} \int m_{01}(dx, dy) \left(\int F dm_{\cdot | 01}^{xy} + \eta \calH(m_{\cdot | 01}^{xy} \mid\mid \pi^{\eta, xy}_{\cdot | 01})\right) \\
    & \qquad + \eta \calH(m_{01} \mid\mid \pi_{\eta, 01}). \\
\end{split}
\end{equation*}

\subsection{Laplace upper bound}\label{sec:laplace-upper-bound}

We have that $\pi^{\eta, xy}_{\cdot | 01} = \calL(X^{\eta, xy}) = R^{\eta, xy}$.
We now show that for small enough $\eta$, by the Laplace principle for $R^{\eta, xy}_{\cdot | 01}$, we have the Laplace principle upper bound: for any $\varepsilon > 0$, 
\begin{align*}
-\eta \log \bE\exp -\frac{1}{\eta} F(X^{\eta, xy}) + \varepsilon \geq \inf_{\varphi \in \calC_1} F(\varphi) + I^{xy}(\varphi) \eqcolon \tilde{F}(x,y).
\end{align*}
By Proposition \ref{prop:F-tilde-continuous}, $\tilde{F}: \bR^{2d} \to \bR$ is continuous and bounded.
Also, since $\pi_{\eta, 01}$ satisfies an LDP with rate function $I_S$, we know that for small enough $\eta$,
\[
-\eta \log \int \exp \left(-\frac{1}{\eta} \tilde{F}\right) d\pi_{\eta, 01} + \varepsilon \geq \inf_{(x,y) \in \bR^{2d}} [\tilde{F}(x,y) + I_S(x,y)].
\]
Combining the two, we have

\begin{equation}\label{eq:Laplace-UB-Full-LDP-P1}
\begin{split}
    -\eta &\log \bE \exp -\frac{1}{\eta} F(X^\eta) \\
    &= \inf_{m_{01} \in \calP(\bR^{2d})} \inf_{m_{\cdot | 01}} \int m_{01}(dx, dy) \left(\int F dm_{\cdot | 01}^{xy} + \eta \calH(m_{\cdot | 01}^{xy} || \pi^{\eta, xy}_{\cdot | 01})\right) + \eta \calH(m_{01} || \pi_{\eta, 01}) \\
    &\geq \inf_{m_{01} \in \calP(\bR^{2d})} \int m_{01}(dx, dy) {\left(\inf_{m_{\cdot | 01}^{xy}} \int F dm_{\cdot | 01}^{xy} + \eta \calH(m_{\cdot | 01}^{xy} || \pi^{\eta, xy}_{\cdot | 01})\right)}
    + \eta \calH(m_{01} || \pi_{\eta, 01}) \\
    &= \inf_{m_{01} \in \calP(\bR^{2d})} \int m_{01}(dx, dy) \left(-\eta \log \bE\exp -\frac{1}{\eta} F(X^{\eta, xy})\right) + \eta \calH(m_{01} || \pi_{\eta, 01}) \\
    &= \inf_{m_{01} \in \calP(\bR^{2d})} \int m_{01}(dx, dy) \tilde F_\eta(x, y) + \eta \calH(m_{01} || \pi_{\eta, 01}), \\
\end{split}
\end{equation}
where in the last step we have defined $\tilde F _\eta (x,y) = -\eta \log \bE [e^{- F(X^{\eta, xy})/\eta}]$. 
Let $F^\varepsilon$ be a modification of $F$, such that its infimum is attained on set $A^\varepsilon \subseteq \calC_1$ whose projection on $t = 0,1$, $A^\varepsilon_{01}$ is compact and contains $\spt \mu \times \spt \nu$, and $\norm{F - F^\varepsilon}_\infty \leq \varepsilon$. %
Using that
\[
    \frac{\int e^f}{\int e^g} \leq e^{\Vert f-g \Vert_\infty},
\]
we have that $\norm{\tilde F_\eta - \tilde F^\varepsilon_\eta} \leq \varepsilon$, and specifically $\tilde F_\eta \geq \tilde F^\varepsilon_\eta - \varepsilon$ for all $\eta > 0$.
Inserting this into the last expression in \eqref{eq:Laplace-UB-Full-LDP-P1} yields
\begin{equation}\label{eq:Laplace-UB-Full-LDP-P1.1}
\begin{split}
    -\eta &\log \bE \exp -\frac{1}{\eta} F(X^\eta) \\
    &\geq \inf_{m_{01} \in \calP(\bR^{2d})} \int m_{01}(dx, dy) \tilde F^\varepsilon_\eta(x, y) + \eta \calH(m_{01} || \pi_{\eta, 01}) - \varepsilon. \\
\end{split}
\end{equation}
Since $\spt \pi_{\eta, 01} = \spt \mu \times \spt \nu \subseteq A^\varepsilon_{01}$, where the infimum of $F^\varepsilon_\eta(x, y)$ is also attained, the selection of $m_{01}$ may now be constrained to measures supported in $A^\varepsilon_{01}$.
By Theorem \ref{thm:uniform-laplace-principle}, for all small enough $\eta$ we have that \[\sup_{x,y \in A^\varepsilon_{01}} |\tilde F^\varepsilon_\eta(x, y) - \tilde F^\varepsilon(x, y)| \leq \varepsilon.\]
Thus, we have,
\begin{equation*}%
\begin{split}
    -\eta &\log \bE \exp -\frac{1}{\eta} F(X^\eta) \\
    &\geq \inf_{m_{01} \in \calP(\bR^{2d})} \int m_{01}(dx, dy) \tilde F^\varepsilon(x, y) + \eta \calH(m_{01} || \pi_{\eta, 01}) - 2\varepsilon \\
    &\geq -\eta \log \int \exp \left(-\frac{1}{\eta} \tilde{F}^\varepsilon\right) d\pi_{\eta, 01} - 2\varepsilon \\
\end{split}
\end{equation*}
Since this holds for arbitrarily small $\varepsilon > 0$ (for small enough $\eta$), we have, using the Laplace principle of the static SB sequence $\{\pi_{\eta, 01}\}_{\eta > 0}$,
\begin{equation*}
\begin{split}
    \liminf_{\eta \downarrow 0} -\eta \log \bE \exp -\frac{1}{\eta} F(X^\eta) %
    &\geq \liminf_{\eta \downarrow 0} -\eta \log \int \exp \left(-\frac{1}{\eta} \tilde{F}^\varepsilon\right) d\pi_{\eta, 01} - 2\varepsilon \\
    &= \inf_{x,y \in \bR^d} [\tilde F^\varepsilon(x,y) + I_S(x,y)] - 2\varepsilon,
\end{split}
\end{equation*}
Because this is true for any $\varepsilon$, and $\norm{\tilde F_\eta - \tilde F^\varepsilon_\eta}$ can be set arbitrarily small, we have
\begin{equation}
\label{eq:upperId}
\begin{split}
    \liminf_{\eta \downarrow 0} -\eta \log \bE \exp -\frac{1}{\eta} F(X^\eta) %
    &\geq \inf_{x,y \in \bR^d} \left\{\tilde F(x,y) + I_S(x,y)\right\}.
\end{split}
\end{equation}

\noindent The right-hand side is:

\begin{equation*}\label{eq:IBxy_IS_to_ID}
\begin{split}
    \inf_{(x,y) \in \bR^{2d}} \left\{\tilde{F}(x,y) + I_{S}(x,y) \right\} &= \inf_{(x,y) \in \bR^{2d}} \left\{ \inf_{\varphi \in \calC_1} \left[ F(\varphi) + I_B^{xy}(\varphi)\right] + I_{S}(x,y) \right\}\\
    &= \inf_{\varphi \in \calC_1} \left\{ F(\varphi) + I_{S}(\varphi_0,\varphi_1) + I_B^{\varphi_0 \varphi_1}(\varphi)\right\}, \\
\end{split}
\end{equation*}
using the fact that $I_B(x,y)(\varphi) = \infty$ whenever $(x,y) \neq (\varphi_0, \varphi_1)$ in the second equality.
Therefore, defining $I_D$ by 
\begin{equation*}
    I_{D}(\varphi) \coloneq I_B^{\varphi_0 \varphi_1}(\varphi) + I_{S}(\varphi_0,\varphi_1),
\end{equation*}
we see that \eqref{eq:upperId} is the Laplace upper bound with $I_D$ as its rate function.

\subsection{Laplace lower bound}
\label{sec:laplace-lower-bound}

To prove the Laplace lower bound, we use arguments similar to those for the Laplace upper bound. We first show that for any $\varepsilon > 0$, for small enough $\eta$, 
\[-\eta \log \bE\exp -\frac{1}{\eta} F(X^{\eta, xy}) - \varepsilon \leq \tilde{F}(x,y).\]
Assume again that $\tilde{F}: \bR^{2d} \to \bR$ is continuous and bounded.
Then for small enough $\eta$, 
\[
-\eta \log \int \exp \left(-\frac{1}{\eta} \tilde{F}\right) d\pi_{\eta, 01} - \varepsilon \leq \inf_{(x,y) \in \bR^{2d}} [\tilde{F}(x,y) + I_S(x,y)].\]
The key additional step is in the first inequality, motivated as follows: for any selection of $m_{01}$, we may select the specific kernel $m_{\cdot | 01}^{*}: \bR^{2d} \times \calB(\calC_1)$ so that $m_{\cdot | 01}^{*, xy}$ has density $\propto e^{-\frac{1}{\eta}F}$ with respect to $\pi^{\eta, xy}_{\cdot | 01}$.
This actually attains the infimum but we keep it as an inequality since that is enough for the proof.
We obtain,

\begin{equation}\label{eq:Laplace-LB-Full-LDP-P1}
\begin{split}
    -\eta &\log \bE \exp -\frac{1}{\eta} F(X^\eta) \\
    &= \inf_{m_{01} \in \calP(\bR^{2d})} \inf_{m_{\cdot | 01}} \int m_{01}(dx, dy) \left(\int F dm_{\cdot | 01}^{xy} + \eta \calH(m_{\cdot | 01}^{xy} || \pi^{\eta, xy}_{\cdot | 01})\right) \\
    & \qquad \quad + \eta \calH(m_{01} || \pi_{\eta, 01}) \\
    &\leq \inf_{m_{01} \in \calP(\bR^{2d})} \int m_{01}(dx, dy) \left(\int F dm_{\cdot | 01}^{*, xy} + \eta \calH(m_{\cdot | 01}^{*,xy} || \pi^{\eta, xy}_{\cdot | 01})\right) \\
    & \qquad \quad + \eta \calH(m_{01} || \pi_{\eta, 01}) \\
\end{split}
\end{equation}
The inner integral is minimized for all $x, y$ by the choice $m_{\cdot | 01}^{*,xy}$, i.e.,
\begin{equation*}
    \int F dm_{\cdot | 01}^{*, xy} + \eta \calH(m_{\cdot | 01}^{*,xy} || \pi^{\eta, xy}_{\cdot | 01}) = \inf_{\mu} \int F d\mu + \eta \calH(\mu || \pi^{\eta, xy}_{\cdot | 01})
\end{equation*}
Using the standard representation formula again, this is further equal to \begin{equation*}\label{eq:laplace-LB-userepr}
    -\eta \log \bE\exp -\frac{1}{\eta} F(X^{\eta, xy}).
\end{equation*}
Now, we apply the same trick as in Section \ref{sec:laplace-upper-bound}, and replace $F$ by $F^\varepsilon$, whose infimum is attained on set $A^\varepsilon \subseteq \calC_1$ whose projection on $t = 0,1$, $A^\varepsilon_{01}$ is compact and contains $\spt \mu \times \spt \nu$, and $\norm{F - F^\varepsilon}_\infty \leq \varepsilon$.
Using the conclusions from the last two displays, and that \[\norm{-\eta \log \bE\exp -\frac{1}{\eta} F(X^{\eta, xy}) - -\eta \log \bE\exp -\frac{1}{\eta} F^\varepsilon(X^{\eta, xy})}_\infty \leq \varepsilon,\]%
we have,
\begin{equation*}%
\begin{split}
    -\eta &\log \bE \exp -\frac{1}{\eta} F(X^\eta) \\
    &\leq \inf_{m_{01} \in \calP(\bR^{2d})} \int m_{01}(dx, dy) \left(-\eta \log \bE\exp -\frac{1}{\eta} F^\varepsilon(X^{\eta, xy})\right) + \eta \calH(m_{01} || \pi_{\eta, 01}) + \varepsilon. \\
\end{split}
\end{equation*}
As in Section \ref{sec:laplace-upper-bound}, the infimization over $m_{01}$ can now be taken over measures supported in $A^\varepsilon_{01}$, and for all small enough $\eta$,
\[\sup_{x,y \in A^\varepsilon_{01}} |\tilde F^\varepsilon_\eta(x, y) - \tilde F^\varepsilon(x, y)| \leq \varepsilon.\]
Thus, we have
\begin{equation*}%
\begin{split}
    -\eta &\log \bE \exp -\frac{1}{\eta} F(X^\eta) \\
    &\leq \inf_{m_{01} \in \calP(\bR^{2d})} \int m_{01}(dx, dy) \left(\tilde F^\varepsilon(x, y) + 
    \varepsilon \right) + \eta \calH(m_{01} || \pi_{\eta, 01}) + \varepsilon \\
    &= -\eta \log \bE \exp -\frac{1}{\eta} \tilde F^\varepsilon(X^\eta) + 2\varepsilon,
\end{split}
\end{equation*}
and then, applying the Laplace principle of the static sequence,
\begin{equation*}
\begin{split}
    \limsup_{\eta \downarrow 0} -\eta \log \bE \exp -\frac{1}{\eta} F(X^\eta) %
    &\leq \limsup_{\eta \downarrow 0} -\eta \log \int \exp \left(-\frac{1}{\eta} \tilde{F}^\varepsilon\right) d\pi_{\eta, 01} + 2\varepsilon \\
    &= \inf_{x,y \in \bR^d} [\tilde F^\varepsilon(x,y) + I_S(x,y)] + 2\varepsilon.
\end{split}
\end{equation*}
Similar to before, since this is true for any $\varepsilon$, and $\norm{\tilde F_\eta - \tilde F^\varepsilon_\eta}$ can be set arbitrarily small, we have the Laplace lower bound
\begin{equation*}
\begin{split}
    \limsup_{\eta \downarrow 0} -\eta \log \bE \exp -\frac{1}{\eta} F(X^\eta) %
    &= \inf_{x,y \in \bR^d} [\tilde F^\varepsilon(x,y) + I_S(x,y)].
\end{split}
\end{equation*}
where the right-hand side is again given by
\begin{equation*}
    \inf_{\varphi \in \calC_1} F(\varphi) + I_{S}(\varphi_0,\varphi_1) + I_B^{\varphi_0 \varphi_1}(\varphi) = \inf_{\varphi \in \calC_1} F(\varphi) + I_D(\varphi).
\end{equation*}
This completes the proof of the lower bound and identification fo the rate function.

\section{Examples of reference dynamics}
\label{sec:case}
To illustrate the derived Laplace principle, conclude the paper with two explicit examples of reference dynamics that satisfy the assumptions used in our proof: scaled Brownian motion (Section \ref{sec:examples_BM}) and Ornstein-Uhlenbeck processes (Section \ref{sec:examples_OU}). We emphasize that this is by no means an exhaustive list and one area of future work is to consider more examples of reference dynamics.

\subsection{Scaled Brownian motion}
\label{sec:examples_BM}
Arguably the most important case, which was the basis for Schr\"odinger's original question and has been analyzed in \cite{Kato24}, where we can verify our assumptions is when the reference dynamics is that of a scaled Brownian motion $W^\eta = \sqrt{\eta} W$ (i.e., \eqref{eq:R-eta-SDE} with $b = 0$ and $\sigma = 1$).
We have then that the transition density is given by
\begin{equation}
    p_\eta(s, x; t, y) = (2\pi\eta(t-s))^{-d/2} \exp\left(-\frac{|y - x|^2}{2\eta (t-s)}\right)
\end{equation}
and the Doob $h$-transform term is %
\begin{equation}\label{eq:BB-h-transform}
    g_\eta^y(s, x) \coloneq (\sqrt{\eta})^2 \nabla_x \log p_\eta(s, x; 1, y) = -\frac{x - y}{1 - s}.
\end{equation}
Then, the Brownian bridges, $W^{\eta, xy}$, follow the familiar dynamics
\begin{equation}\label{eq:BB-dynamics}
    dW^{\eta, xy}_t = -\frac{W^{\eta, xy}_t - y}{1 - s} dt + \sqrt{\eta} dW_t.
\end{equation}

We now want to establish that assumptions \ref{ass:bounded-lipshitzness-of-log-gradients}, \ref{ass:exp-boundedness}, and \ref{ass:log-gradient-convergence} hold.
Assumption \ref{ass:bounded-lipshitzness-of-log-gradients} is readily verified from \eqref{eq:BB-h-transform}.
Assumption \ref{ass:exp-boundedness} is shown, using \eqref{eq:classical-estimate-d} and classical Gaussian estimates (alternatively, using the goodness of the rate function in Schilder's theorem); we omit the proof for the sake of brevity. %
For Brownian bridges, $g_\eta^y(s, x) = g^y(s, x)$ is independent of $\eta$ and Assumption \ref{ass:log-gradient-convergence} is thus satisfied. 

The $\nu$-controlled limit dynamics of \eqref{eq:BB-dynamics} is
\begin{equation}\label{eq:controlled-BB-limit}
    d\bar{W}^{xy}_t = %
    (b(t, \bar{W}^{xy}_t) + g^y(t, \bar{W}^{xy}_t) + \nu_t) dt = 
    -\frac{\bar{W}^{xy}_t - y}{1 - s} dt + \nu_t dt, \ \ \bar{W}^{xy}_0 = x.
\end{equation}
and the rate function becomes
\begin{equation*}\label{eq:rate-function-BB-xy}
    I_B^{xy}(\varphi) = \inf_{\nu \in U^{xy}_\varphi} \frac{1}{2}\int_0^1|\nu_s|^2ds,
\end{equation*}
where $U^{xy}_\varphi$ are the set of controls driving noise-free controlled dynamics \eqref{eq:controlled-BB-limit} to $\varphi$.

The form \eqref{eq:rate-function-BB-xy} of the rate function is what one typically obtains from the weak convergence approach, due to the reliance on controlled processes. We can also verify that this formulation is 
equal to the one given in \cite{Kato24}, namely,
\begin{equation*}
    I_B^{xy}(\varphi) = \frac{1}{2}\int_0^1 |\varphi'_t|^2 dt - \frac{1}{2}|x - y|^2,
\end{equation*}
where the first term is interpreted as $+\infty$ for any non-absolutely continuous $\varphi \in \calC_1$, and the second is precisely $c$.
This gives the rate function $I_D$ for dynamic Schr\"odinger bridges with scaled Brownian motion as reference dynamics as
\begin{equation}
\begin{split}
    I_D(\varphi) &= (c-(-\psi \oplus \phi))(\varphi_0, \varphi_1) + \frac{1}{2}\int_0^1 |\varphi'_t|^2 dt - \frac{1}{2}|\varphi_0 - \varphi_1|^2 \\
    &= -(-\psi \oplus \phi))(\varphi_0, \varphi_1) + \frac{1}{2}\int_0^1 |\varphi'_t|^2 dt,
\end{split}
\end{equation}
which is precisely the one obtained in \cite{Kato24}.

\subsection{Ornstein-Uhlenbeck processes}
\label{sec:examples_OU}
The Ornstein-Uhlenbeck (OU) process is defined by the dynamics
\begin{equation*}\label{eq:OU-dynamics}
\begin{split}
    dX_t &= -\theta X_t dt + \sqrt{\eta} dW_t.
\end{split}
\end{equation*}
The corresponding transition density is given by 
\begin{equation*}
    p_\eta(s, x; t, y) = \calN\left(y; e^{-\theta (t-s)} x, \eta \frac{1 - e^{-2 \theta (t-s)}}{2 \theta}\right),
\end{equation*}
and the Doob $h$-transform term is
\begin{equation}\label{eq:OU-h-transform}
    g_\eta^y(s, x) \coloneq %
    \frac{2 \theta}{1 - e^{-2 \theta (1-s)}} (y - e^{-\theta (1-s)} x).
\end{equation}
Assumptions \ref{ass:bounded-lipshitzness-of-log-gradients}, \ref{ass:exp-boundedness} %
and \ref{ass:log-gradient-convergence} follow in precisely the same fashion as for the Brownian bridges. Thus, we have that the dynamic Schr\"odinger bridges with an OU-process as reference dynamics satisfy the large deviation principle of Theorem \ref{thm:LDP_dSB}.

\section*{Acknowledgment}
The research of VN and PN was supported by Wallenberg AI, Autonomous Systems and Software Program (WASP) funded by the Knut and Alice Wallenberg Foundation. The research of PN was also supported by the Swedish Research Council (VR-2018-07050, VR-2023-03484).

{
\bibliographystyle{plain}
\hypersetup{hidelinks}
\bibliography{references}
}

\appendix

\end{document}